\documentclass[a4paper,11pt,reqno]{amsart}
\usepackage{amssymb,latexsym,amsthm,amsmath,amsfonts,mathrsfs}
\usepackage{caption,color,graphicx}

\usepackage[numbers,comma,square,sort&compress]{natbib}
\usepackage{hyperref}
\usepackage[centering, top=3cm, bottom=2.5cm]{geometry}
\usepackage{bbm}
\usepackage{enumitem}
\usepackage{verbatim}

\usepackage{mathtools}

\usepackage{stackengine}

\usepackage{xcolor}

\usepackage{xcolor}

\usepackage{braket}
\usepackage[normalem]{ulem}
\newcommand{\N}{\mathbb{N}}
\newcommand{\Z}{\mathbb{Z}}
\newcommand{\Q}{\mathbb{Q}}
\newcommand{\R}{\mathbb{R}}
\newcommand{\C}{\mathbb{C}}
\newcommand{\id}{\mathrm{id}}
\newcommand{\A}{\mathbf{A}}
\newcommand{\x}{\mathbf{x}}
\newcommand{\y}{\mathbf{y}}

\newcommand{\Crop}{\mathrm{Crop}}
\newcommand{\FF}{\mathcal{F}}
\newcommand{\ZZ}{\mathcal{Z}}
\newcommand{\Bin}{\mathrm{Bin}}
\newcommand{\YD}{\mathrm{YD}}
\newcommand{\len}{\mathrm{len}}
\newcommand{\hgt}{\mathrm{ht}}
\newcommand{\LL}{\mathbf{L}}
\newcommand{\II}{\mathbf{I}}
\newcommand{\is}{\mathrm{lis}}
\newcommand{\BB}{\mathcal{B}}
\newcommand{\Hom}{\mathrm{Hom}}
\newcommand{\M}{\mathcal{M}}

\newcommand{\ii}{\mathbf{i}}
\newcommand{\jj}{\mathbf{j}}
\newcommand{\calM}{\mathcal{M}}
\newcommand{\calN}{\mathcal{N}}

\makeatletter
\def\old@comma{,}
\catcode`\,=13
\def,{%
  \ifmmode%
    \old@comma\discretionary{}{}{}%
  \else%
    \old@comma%
  \fi%
}
\makeatother

\newcommand{\scb}{\scalebox}
\newcommand\xof[3][1ex]{\ensurestackMath{%
  \setbox0=\hbox{$#2$}%
  \setbox2=\hbox{$#3$}%
  \kern\wd0\kern-\dimexpr#1\relax%
  \stackinset{r}{#1}{b}{\dimexpr#1-1ex}{\mathrlap{#3}}{%
    \stackinset{l}{#1}{t}{\dimexpr#1-1ex}{\mathllap{#2}}{\bigg/}%
  }%
  \kern-\dimexpr#1\relax\kern\wd2%
}}
\theoremstyle{plain}
\newtheorem{theorem}{Theorem}[section]
\newtheorem{lemma}[theorem]{Lemma}
\newtheorem{proposition}[theorem]{Proposition}
\newtheorem{corollary}[theorem]{Corollary}

\theoremstyle{definition}
\newtheorem{definition}[theorem]{Definition}
\newtheorem{remark}[theorem]{Remark}
\newtheorem{example}[theorem]{Example}

\title{Binomial Cayley Graphs and Applications to Dynamics on Finite Spaces}

\author[Bernat Bassols-Cornudella]{Bernat Bassols-Cornudella$^*$}
\address{Department of Mathematics, Imperial College London, London SW7 2AZ, UK}
\thanks{$^*$Department of Mathematics, Imperial College London, London SW7 2AZ, UK}
\email{\href{mailto:bernat.bassols-cornudella20@imperial.ac.uk}{bernat.bassols-cornudella20@imperial.ac.uk}}

\author[Francesco Viganò]{Francesco Viganò$^{*}$}
\address{Department of Mathematics, Imperial College London, London SW7 2AZ, UK}
\email{\href{mailto:f.vigano21@imperial.ac.uk}{f.vigano21@imperial.ac.uk}}

\thanks{To appear in Algebraic Combinatorics (ISSN: 2589-5486), \url{http://algebraic-combinatorics.org/}}
\thanks{Manuscript submitted 18th May 2023 and accepted 4th February 2024.}

\begin{document}
\begin{abstract}
    \emph{Binomial Cayley graphs} are obtained by considering the binomial coefficient of the weight function of a given Cayley graph and a natural number. We introduce these objects and study two families: one associated with symmetric groups and the other with powers of cyclic groups. We determine various combinatorial properties of these graphs through the spectral analysis of their adjacency matrices. In the case of symmetric groups, we establish a relation between the multiplicity of the null eigenvalue and longest increasing sub-sequences of permutations by means of the RSK correspondence. Finally, we consider dynamical arrangements of finitely many elements in finite spaces, which we refer to as \emph{particle-box systems}. We apply the results obtained on binomial Cayley graphs in order to describe their degeneracy.
\end{abstract}

\maketitle

\section{Introduction}\label{section.introduction}
Let us consider a set of $m$ different boxes and a collection of $n$ labelled particles. On this setup, we play a game against a bot that repeatedly arranges all particles across the different boxes. In general, these particles can be laid out in $m^n$ different configurations, allowing several of them in the same box, or leaving some boxes empty. Before the game begins, the bot fixes once and for all a probability distribution $p$ on the set of all possible $m^n$ configurations and an integer $k\le n$. \\

We play as an external observer interested in the disposition of the $n$ particles. However, on every round the bot only allows us to observe where $k$ particles of our choice have been allocated. That is, the game unfolds over infinitely many independent rounds consisting of:
\begin{enumerate}
    \item We select $k$ of the $n$ particles.
    \item The bot samples an arrangement from $p$ and lays out the $n$ particles as such. 
    \item We are given the location of the $k$ selected particles.
\end{enumerate}

\begin{figure}[htb]
    \centering
    \includegraphics[scale=.45]{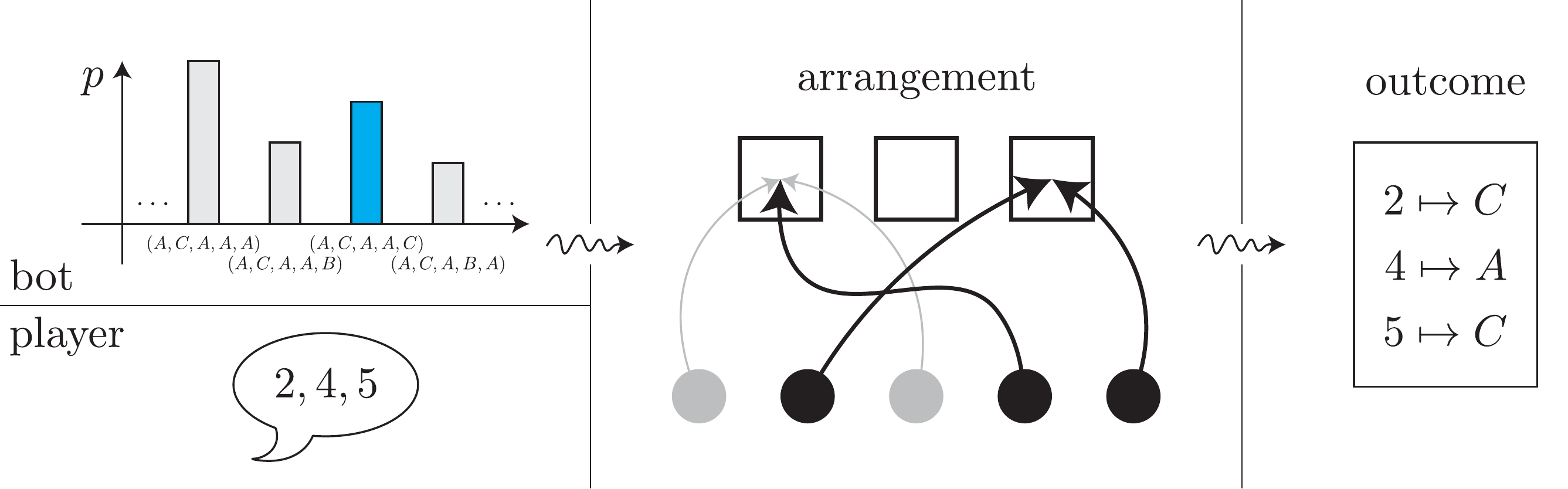}
    \caption{A round of the game: after sampling an arrangement (selected in blue) from $p$, the bot places the particles and reveals the location of those chosen by the player. In this case, $k=3, n=5, m=3$. The player chooses to see particles $2, 4,$ and $5$, that are placed in boxes $C, A$, and $C$, respectively.}
    \label{fig:game_all}
\end{figure}

The round then finishes and a new one begins. We refer to this game as a \emph{particle-box system}. After infinitely many rounds and all different choices of observable particles, by the law of large numbers we know the probability distributions of the bot's moves when arranging any $k$-subset of the $n$ particles, i.e. the $k$-marginal distributions of $p$. We refer to this family of $\binom{n}{k}$ probability distributions as the \emph{$k$-restriction} of $p$ and denote it by $p_{|k}$.\\

We win the game if we are able to recover the original distribution $p$ from its $k$-restriction $p_{|k}$. But is this possible? And if it isn't, how far are we from determining $p$?\\

The work presented herein explores these questions, providing a negative answer to the first: it turns out that we generally cannot win the game. For the second, we note that reducing the number of observed particles from $n$ to $k$ gives rise, in general, to a \emph{degeneracy}: a positive-dimension space of probability distributions on the set of all $m^n$ arrangements resemble $p$ when $k$-restricted to $p_{|k}$. Studying this degeneracy amounts to determining the kernel of the $\left( \binom{n}{k} \cdot m^k \right) \times m^n$ restriction matrix $M$, defined as
\begin{equation}\label{equation.restriction_matrix_Zm^n}
    M_{(\ii, \jj), f} =
    \begin{cases}
        1 &\text{if }  f(i_1, \dots, i_k) = (j_1, \dots, j_k),\\    
        0 &\text{otherwise,}
    \end{cases}
\end{equation}
where we identify the set of boxes with the group $\Z_m = \{0, \dots, m-1\}$ and the possible arrangements of $n$ particles with the maps $f \colon \calN = \{1, \dots, n\} \to \Z_m$.\\

The kernel of $M$ conveniently coincides with the kernel of the $m^n \times m^n$ matrix $A = M^T M$, with entries
\begin{equation*}
    A_{fg} = \binom{\ZZ(g - f)}{k},
\end{equation*}
where $\ZZ(g-f)$ denotes the number of zeros of the map $g-f \colon \calN \to \Z_m$. To our advantage, this matrix can be interpreted as the adjacency matrix of a weighted Cayley graph on the group of maps $\calN \to \Z_m$, isomorphic to $(\Z_m)^n$. Letting $k$ vary, we obtain a family of weighted Cayley graphs.\\

Spectra of (the adjacency matrices of) weighted normal Cayley graphs are completely described in terms of the irreducible characters of their underlying groups (see Theorem \ref{theorem.spectra_weighted_normal_Cayley_graphs}). Using this technique, in Theorem \ref{theorem.eigenvalues_binomial_Cayley_graph_Zm^n} we provide an explicit description of the spectra of weighted Cayley graphs associated with this particle-box system. In particular, the dimension of the kernel of $A$, and equivalently of $M$, is (Corollary \ref{corollary.rank_binomial_Cayley_graphs_Zm^n})
\begin{equation*}
    \sum_{t < n-k} \binom{n}{t} \cdot (m-1)^{n-t}.
\end{equation*}
This provides a closed expression to the recursive formula given in \cite[Theroem 27]{daCosta2022}.\\

We pay particular attention to the case $n = m$, when the system takes the form of the so-called $k$-point motion on finite spaces introduced in \cite{daCosta2022}. In fact, the degeneracy of a particle-box system is linked to the notion of stochastic $n$-point D-bifurcations, whose complexity decreases as more particles $k$ are allowed to be observed. Originally, the study of the $k$-point motion on finite spaces arose from previous results on stochastic flows and particularly on Brownian flows of diffeomorphisms on $\R^d$ given by Baxendale in \cite{Baxendale1984}. In this continuous setting, the analogous of the restriction $p_{|2}$ is sufficient to fully characterise the analogous of $p$ (see e.g. \cite[Chapter 4]{Kunita1990}, and in particular Theorems 4.2.4 and 4.2.5 therein).\\

The jump from $\R^d$ into a discrete state space $\{1, \dots, m\}$ conceptually brings Brownian flows of diffeomorphisms into bijective maps from the space of particles $\calN$ to the space of boxes $\calM$. In fact, since $n = m$, $\calN$ and $\calM$ coincide, and the set of admissible transformations corresponds to the symmetric group $S_m$. Specifically, $p$ is now a probability distribution on $S_m$. In this setting, the 2-point motion no longer uniquely characterises higher order point motions, suggesting a thorough study of the $k$-restriction from $p$ to $p_{|k}$ in the bijective framework.

\begin{figure}[hbt]
    \centering
    \includegraphics[scale=.45]{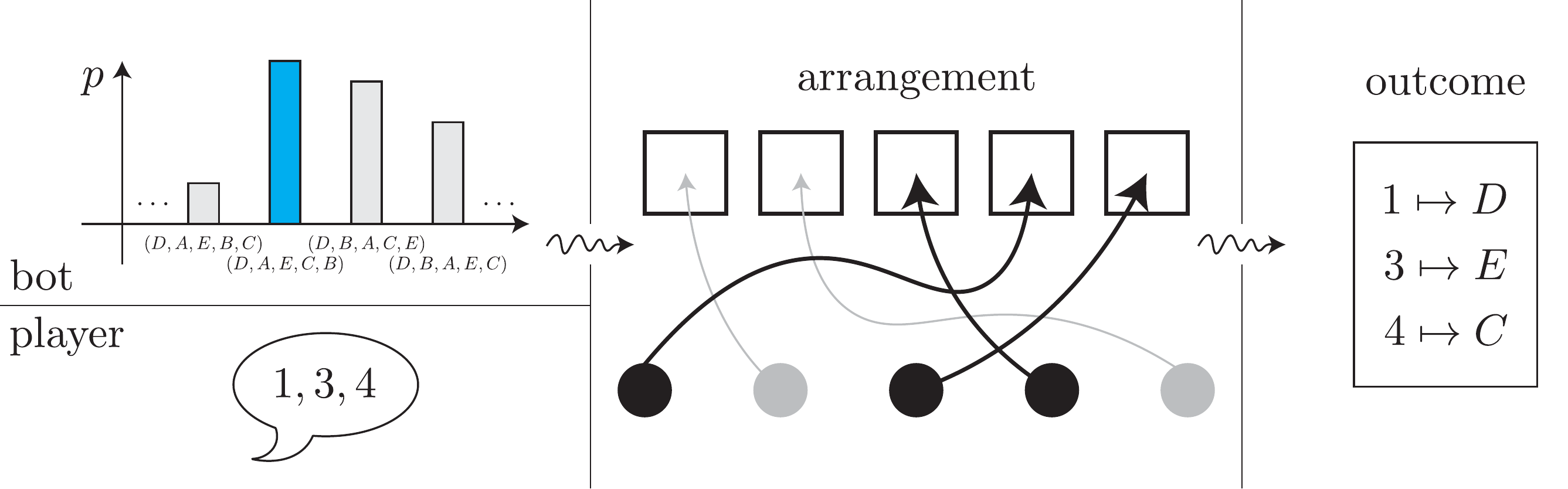}
    \caption{A round of the game in the bijective case: after sampling an arrangement (selected in blue) from $p$, the bot places the particles and reveals reveals the location of those chosen by the player. In this case, $k=3, n=m=5$. The player chooses to see particles $1, 3,$ and $4$, that are placed in boxes $D, E$, and $C$, respectively.}
    \label{fig:game_bij}
\end{figure}

In the bijective case, the restriction matrix $M$ takes the same form as in Equation \eqref{equation.restriction_matrix_Zm^n}. Its size can be reduced to $\left(\binom{m}{k} \cdot \frac{m!}{(m-k)!} \right) \times \binom{m}{k}$, and is indexed as $M_{(\ii, \jj), \sigma}$, where $\sigma \in S_m$. Again, the kernel of $M$ coincides with the kernel of the $m! \times m!$ matrix $A = M^T M$, with entries
\begin{equation*}
    A_{\sigma\tau} = \binom{\FF(\tau\sigma^{-1})}{k},
\end{equation*}
where $\FF(\tau\sigma^{-1})$ denotes the number of fixed points of the permutation $\tau\sigma^{-1} \in S_m$. Analogously, this matrix is the adjacency matrix of a weighted Cayley graph on $S_m$. In Theorem \ref{theorem.eigenvalues_binomial_Cayley_graph_Sm} we determine the spectrum of this graph. As a consequence, the dimension of the kernel of $A$, and equivalently of $M$, is (Corollary \ref{corollary.rank_binomial_Cayley_graphs_Sm})
\begin{equation*}
    \sum_{\substack{\mu \vdash m \\ \mu_1 < m-k}} \chi^\mu(\id)^2,
\end{equation*}
where $\chi^\mu(\id)$ is the degree of the irreducible character associated with $\mu$. The value $\chi^\mu(\id)$ coincides with the number of standard Young tableaux of shape $\mu$. By means of the RSK correspondence, we can rephrase this result as: the dimension of the kernel of $A$ is the number of permutations in $S_m$ with no increasing sub-sequences of length $m-k$ (Corollary \ref{corollary.increasing_sequences}). This also proves the formula for the degeneracy of the $k$-point motion conjectured in \cite[Section 3.3.2]{daCosta2022}.\\

The two families of Cayley graphs presented above are instances of graphs arising from a more general construction. In this work we refer to them as \emph{binomial Cayley graphs}, for which the weight function is obtained by taking the binomial coefficient of an original weight function and a natural number $k$.\\

{\bf Outline.} In Section \ref{section.weighted_normal_Cayley_graphs_and_binomial_Cayley_graphs}, we recall the definition of weighted normal Cayley graphs, state a result on their eigenvalues (Theorem \ref{theorem.spectra_weighted_normal_Cayley_graphs}), and introduce binomial Cayley graphs. In Section \ref{section.binomial_Cayley_graphs_on_Sm}, we study a family of binomial Cayley graphs on symmetric groups and describe their spectrum (Theorem \ref{theorem.eigenvalues_binomial_Cayley_graph_Sm}). We also obtain a link between the dimension of the null eigenvalue to increasing sub-sequences through the RSK correspondence (Corollary \ref{corollary.increasing_sequences}). Section \ref{section.binomial_Cayley_graphs_on_Zm^n} is structured analogously to Section \ref{section.binomial_Cayley_graphs_on_Sm}. We analyse a family of binomial Cayley graphs on powers of cyclic groups and determine their spectrum (Theorem \ref{theorem.eigenvalues_binomial_Cayley_graph_Zm^n}). In Section \ref{section.degeneracies}, we consider particle-box systems and describe their degeneracy (Theorem \ref{thm.degeneracy-k-hom}). As an application of our results on binomial Cayley graphs,  we determine the degeneracy of particle-box systems for specific choices of admissible functions, related to powers of cyclic groups (Corollary \ref{cor.degeneracy_random_maps}) and symmetric groups (Corollary \ref{cor.degeneracy_random_bijections}). We conclude the paper with some final remarks in Section \ref{sec.final}.

\section{Weighted normal Cayley graphs and binomial Cayley graphs}\label{section.weighted_normal_Cayley_graphs_and_binomial_Cayley_graphs}

In this section we introduce weighted normal Cayley graphs and binomial Cayley graphs. We refer to \cite[Section 2]{Babai1979} and \cite[Section 3.7]{Godsil2001} for a general discussion on Cayley graphs.\\

Let $G$ be a finite group and $\omega \colon G \to \C$. The (directed) \emph{weighted Cayley graph} $\Gamma(G, \omega)$ has as set of vertices the elements of $G$ and, for $g, h \in G$, the edge joining $g$ to $h$ has complex weight $\omega(hg^{-1})$ (see Figure \ref{fig:weighted_cayley_graph_1}). Its adjacency matrix $\A(G, \omega)$ is a $|G| \times |G|$ matrix whose $(g, h)$ entry is $\omega(hg^{-1})$. Self-loops are admitted.

\begin{figure}[htb]
    \centering
    \includegraphics[scale=0.38]{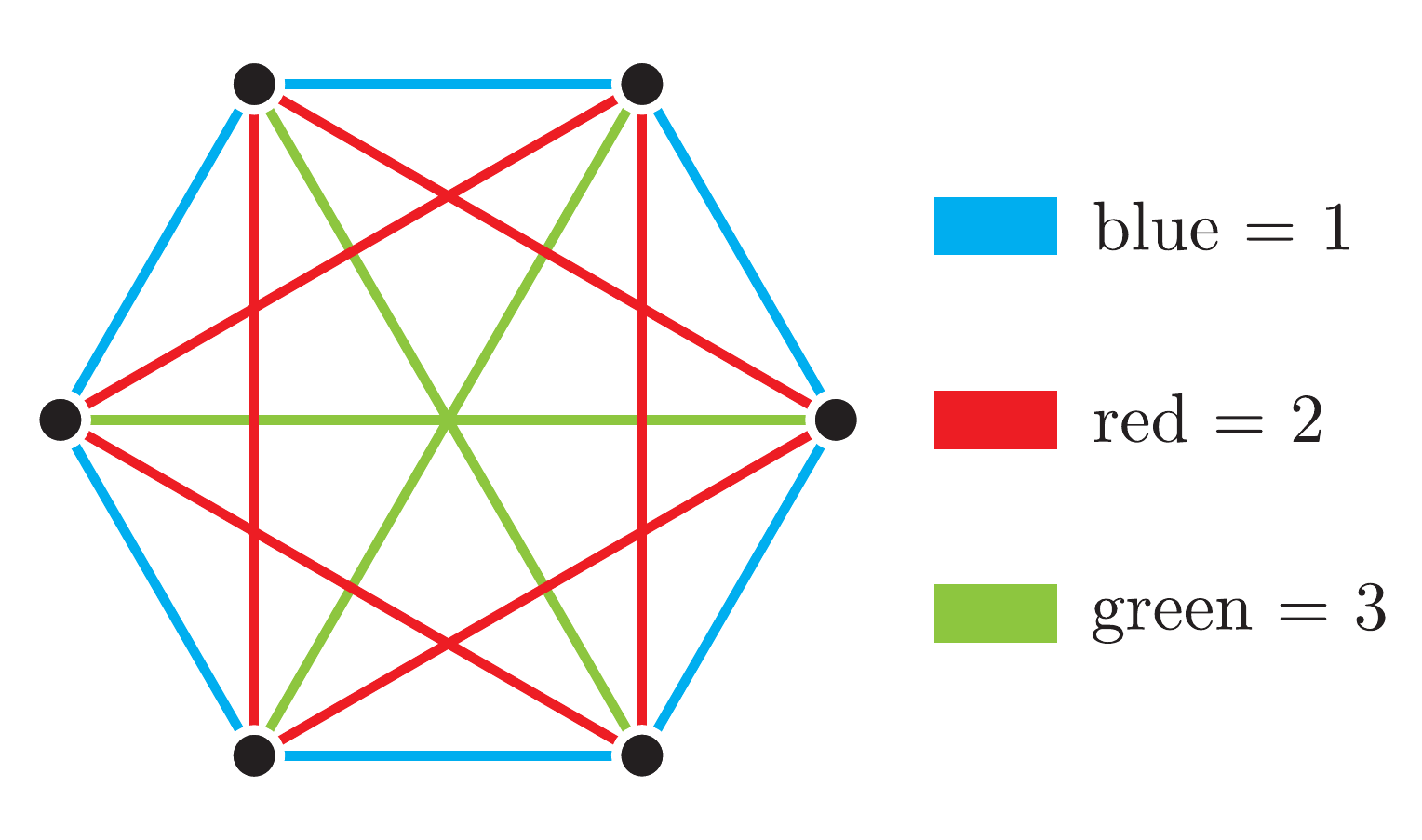}
    \caption{Example of weighted Cayley graph on the cyclic group $\Z_6$. The weight function takes values $\omega(\pm g) = g$ for $g = 0, 1, 2, 3$. The symmetry of $\omega$ makes the graph undirected.}
    \label{fig:weighted_cayley_graph_1}
\end{figure}

The group $G$ acts transitively on the set of nodes of $\Gamma(G, \omega)$ by right multiplication. Each of these node permutations is in fact a graph automorphism since the weight of the edge joining $g$ to $h$ only depends on $hg^{-1}$. It follows that every weighted Cayley graph is regular (that is, each node has the same weighted degree). In particular, the weighted degree $d(G, \omega)$ of each node is
\begin{equation*}
d(G, \omega) = \sum_{g \in G} \omega(g).
\end{equation*}

This paper makes essential use of weight functions. However, we remark that \emph{non-weighted} Cayley graphs can be recovered as particular weighted Cayley graphs. Non-weighted Cayley graphs are defined through a subset $S \subseteq G$. Usually one also assumes that $S$ is invariant under inversion ($S^{-1}=S)$ and that $\id \notin S$ (no self-loops). An edge connects $g$ to $sg$ for every $g \in G$, $s \in S$ (see Figure \ref{fig:simple_cayley_graphs}). This is equivalent to saying that $g$ is connected to $h$ if and only if $hg^{-1} \in S$. A non-weighted Cayley graph can therefore be described by the weight (characteristic) function
\begin{equation*}
    \omega(g) =
    \begin{cases}
        1 & \text{if } g \in S, \\
        0 & \text{if } g \notin S.
    \end{cases}
\end{equation*}

\begin{figure}[htb]
    \centering
    \includegraphics[scale=0.38]{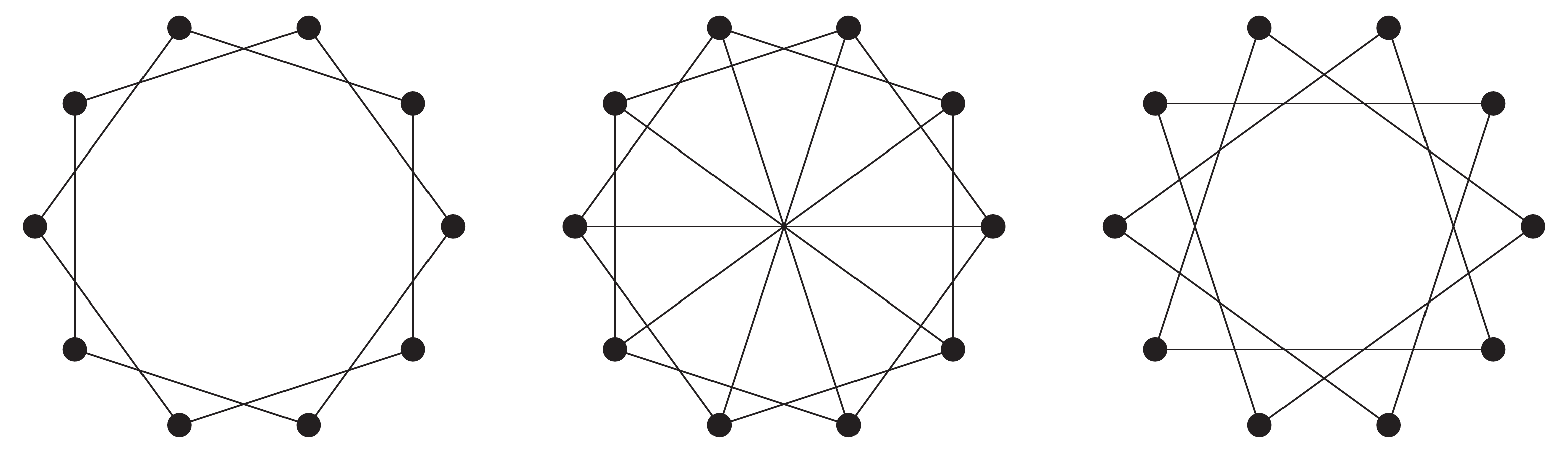}
    \caption{Examples of undirected non-weighted Cayley graphs: the weight function evaluates $1$ on a subset $S \subseteq G$. In this case, $G$ is the cyclic group $\Z_{10}$. From left to right, $S=\{ \pm 2 \}$, $S=\{ \pm 2 , \pm 5 \}$, and $S=\{ \pm 3 \}$. The left graph is not connected, as $\{ \pm 2 \}$ is not a set of generator of $\Z_{10}$. The middle and right graphs are connected, as $\{ \pm 2 , \pm 5 \}$ and $\{ \pm 3 \}$ both generate $\Z_{10}$.}
    \label{fig:simple_cayley_graphs}
\end{figure}

Given a weighted Cayley graph $\Gamma(G, \omega)$ and a $k \in \N$, we define the \emph{binomial Cayley graph} $\Gamma_\Bin(G, \omega, k)$ as $\Gamma(G, \omega_k)$, where $\omega_k$ is defined through the (generalised) binomial coefficient
\begin{equation*}\label{equation.omega_k}
\omega_k(g) = \binom{\omega(g)}{k} = \frac{\omega(g) \cdot \left(\omega(g) - 1\right) \cdots \left( \omega(g) - k + 1 \right)}{k!}.
\end{equation*}
In particular, for $k=1$ we recover the original Cayley graph $\Gamma_\Bin(G, \omega, 1) = \Gamma(G, \omega)$. For $k=0$, $\Gamma_\Bin(G, \omega, 0)$ is the complete graph with self-loops (all edges have equal weight $1$). The graph becomes sparser as $k$ increases, up until no edges are left when $k$ exceeds the maximal value of $\omega$ (see Figure \ref{fig.weighted_cayley_graphs}).

\begin{figure}[htb]
    \centering
    \includegraphics[scale=0.38]{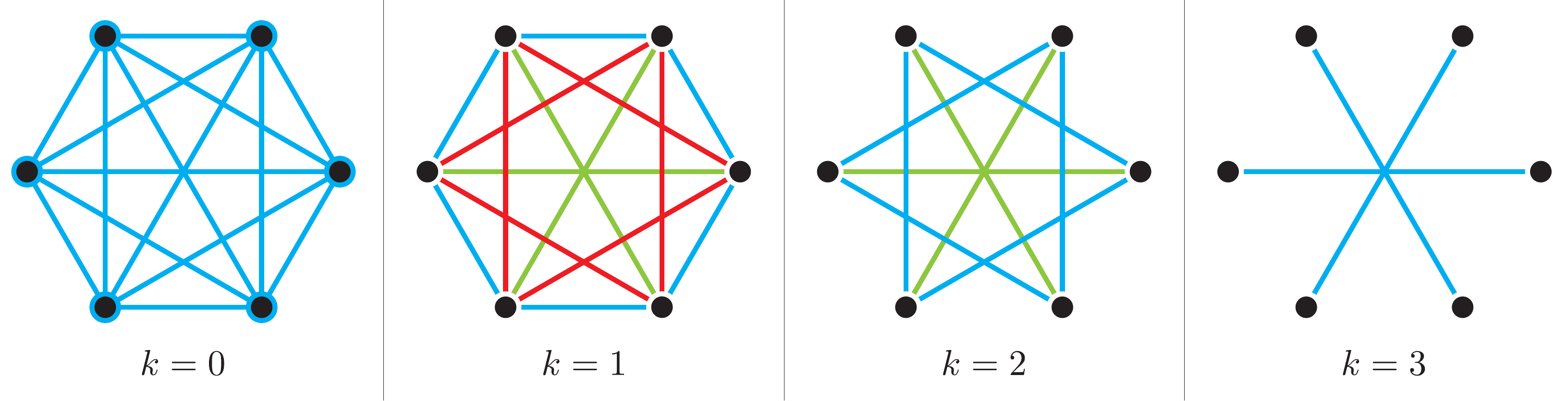}
    \caption{Examples of binomial Cayley graphs obtained from the weighted Cayley graph in Figure \ref{fig:weighted_cayley_graph_1}. Colour encircling nodes represents self-loops. On each edge, the change in colour from left to right follows the binomial coefficient.}
    \label{fig.weighted_cayley_graphs}
\end{figure}

Despite the generality of this construction, in this paper we assume that:
\begin{enumerate}
\item $\omega$ takes values in $\N = \{0, 1, \dots\}$, so that $\A(G, \omega)$ has non-negative integer entries. In particular, $\omega$ is real-valued, or equivalently invariant under complex conjugation.
\item For all $g \in G$, $\omega(g^{-1}) = \omega(g)$, that is, $\omega$ is \emph{invariant under inversion}. This implies that $\A(G, \omega)$ is symmetric and that $\Gamma(G, \omega)$ is an undirected graph.
\item For all $g, h \in G$, $\omega(h^{-1}gh) = \omega(g)$, that is, $\omega$ is a \emph{class function}, or equivalently $\omega$ is invariant under group conjugation.
\end{enumerate}
Under these conditions, we refer to $\Gamma(G, \omega)$ as a weighted \emph{normal} Cayley graph. For instance, the graph in Figure \ref{fig:weighted_cayley_graph_1} is normal. It is worth mentioning that hypothesis (3) implies that the action of $G$ by left multiplication on the nodes of $\Gamma(G, \omega)$ induces graph automorphisms as well.\\

Given a weighted normal Cayley graph $\Gamma(G, \omega)$ and a $k \in \N$, the binomial Cayley graph $\Gamma_\Bin(G, \omega, k)$ is itself a weighted normal Cayley graph. We underline that, by hypothesis (1), the generalised binomial $\binom{\omega(g)}{k}$ is in fact a standard binomial. In particular, $\binom{\omega(g)}{k} = 0$ if $\omega(g) < k$. We also denote by $\A_\Bin(G, \omega, k)$ the adjacency matrix of the binomial Cayley graph $\Gamma_\Bin(G, \omega, k)$. Similarly, we denote by $d_\Bin(G, \omega, k)$ the weighted degree of each of the nodes of $\Gamma_\Bin(G, \omega, k)$. \\

We state here an elegant theorem, fundamental to our discussion, describing the spectra of weighted normal Cayley graphs in terms of the irreducible characters of the group $G$. The theorem descends as a corollary from a more involved result, where the hypothesis of normality is dropped, and which appeared for the first time in \cite[Theorem 3.1]{Babai1979} (see also \cite[Theorem 3]{Diaconis1981}). For completion, in Section \ref{sec.eigenvectors_binomial} we discuss eigenvectors of weighted normal Cayley graphs, for which we refer to  \cite[Section 1]{Rockmore2002} and \cite[Section 4]{Roichman1999}. \\

We recall the standard inner product on the space $\mathscr{F}(G, \C)$ of $\C$-valued functions on $G$. Given $\alpha, \beta \colon G \to \C$, we define
\begin{equation*}\label{equation.inner_product}
\langle \alpha, \beta \rangle = \frac{1}{|G|} \sum_{g \in G} \alpha(g) \overline{\beta(g)}.
\end{equation*}

\begin{theorem}\label{theorem.spectra_weighted_normal_Cayley_graphs}
Assume that $\Gamma(G, \omega)$ is a weighted normal Cayley graph. Then, for every irreducible character $\chi$ of $G$, the value
\begin{equation*}\label{equation.weighted_normal_Cayley_eigenvalue}
\lambda_\chi = \frac{1}{\chi(\id)} \sum_{g \in G} \omega(g) \chi(g) = \frac{|G|}{\chi(\id)} \langle \omega, \overline{\chi}
\rangle
\end{equation*}
is an eigenvalue of $\A(G, \omega)$. Its contribute in multiplicity is $\chi(\id)^2$.
\end{theorem}

\begin{remark}\label{remark.multiplicity_eigenvalue}
Different characters $\chi$ may induce the same eigenvalue $\lambda_\chi$. More explicitly, given a real number $\lambda$, the multiplicity of $\lambda$ as an eigenvalue of $\A(G, \omega)$ is $\sum_{\chi \colon \lambda_\chi = \lambda} \chi(\id)^2$.
\end{remark}

Having introduced the binomial construction of $\omega_k$, a natural question to investigate is whether there exists a formula that relates the spectrum of $\Gamma_\Bin(G, \omega, k)$ to that of $\Gamma(G, \omega)$. We present this in Section \ref{sec.ugly_formula_general_eigenvalues_binomial}, and already point out that this expression does not seem to be easily tractable in general. This supports the specialised analysis of two families of binomial Cayley graphs, which we expose in Sections \ref{section.binomial_Cayley_graphs_on_Sm} and \ref{section.binomial_Cayley_graphs_on_Zm^n}.

\section{Binomial Cayley graphs on symmetric groups}\label{section.binomial_Cayley_graphs_on_Sm}

Let $S_m$ be the symmetric group on $m$ elements $\{1, \dots, m\}$ and consider the weight function $\FF \colon S_m \to \N$, which gives the number of fixed points of a permutation (see Figure \ref{fig:sym_cayley_graph_1}). This induces, for every $k \in \N$, the weight function $\FF_k \colon S_m \to \N$, defined as
\begin{equation*}\label{equation.omega_k_Sm}
\FF_k(\sigma) = \binom{\FF(\sigma)}{k}.
\end{equation*}

As the weight function $\FF$ takes values in $\N$, is invariant under inversion and is a class function, we obtain a family of associated weighted normal binomial Cayley graphs $\Gamma_\Bin(S_m, \FF, k)$, for $k \in \N$ (see Figure \ref{fig:sym_cayley_graphs}).

\begin{figure}[htb]
    \centering
    \includegraphics[scale=0.45]{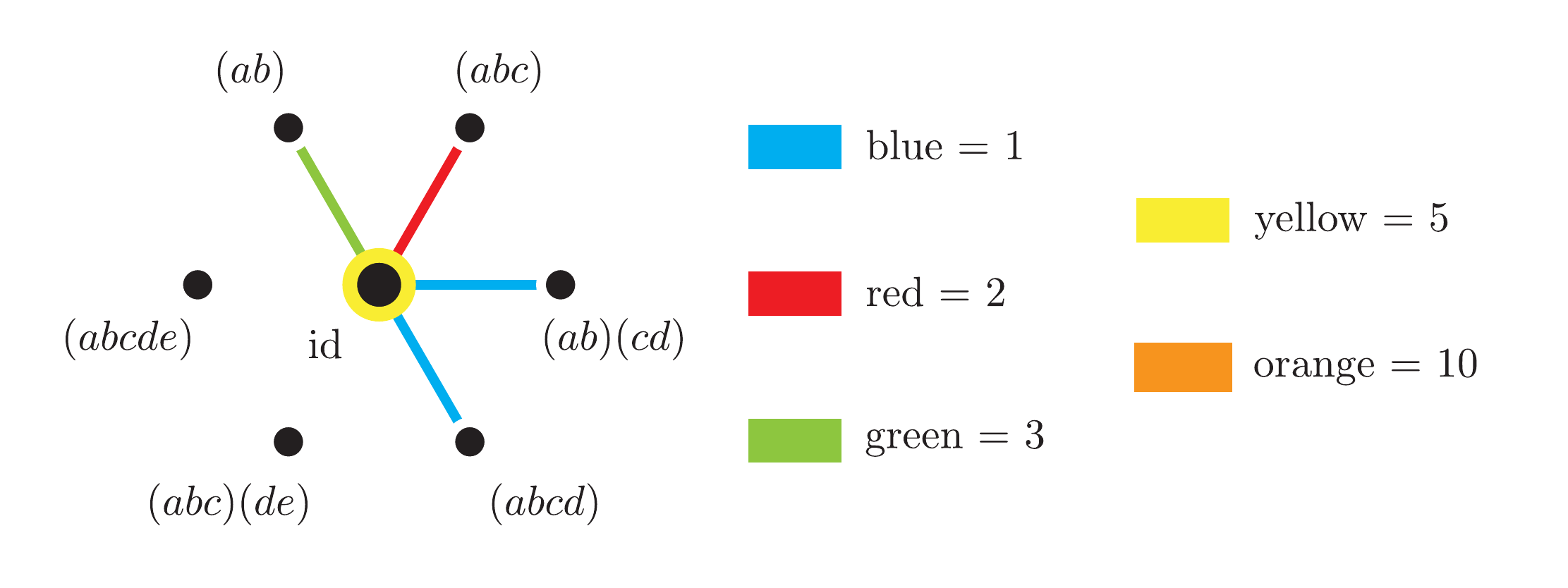}
    \caption{The neighbourhood of the identity of the weighted Cayley graph on the symmetric group $S_5$, induced by the weight function $\FF$. Colour encircling nodes represents self-loops.}
    \label{fig:sym_cayley_graph_1}
\end{figure}

\begin{figure}[htb]
    \centering
    \includegraphics[scale=0.45]{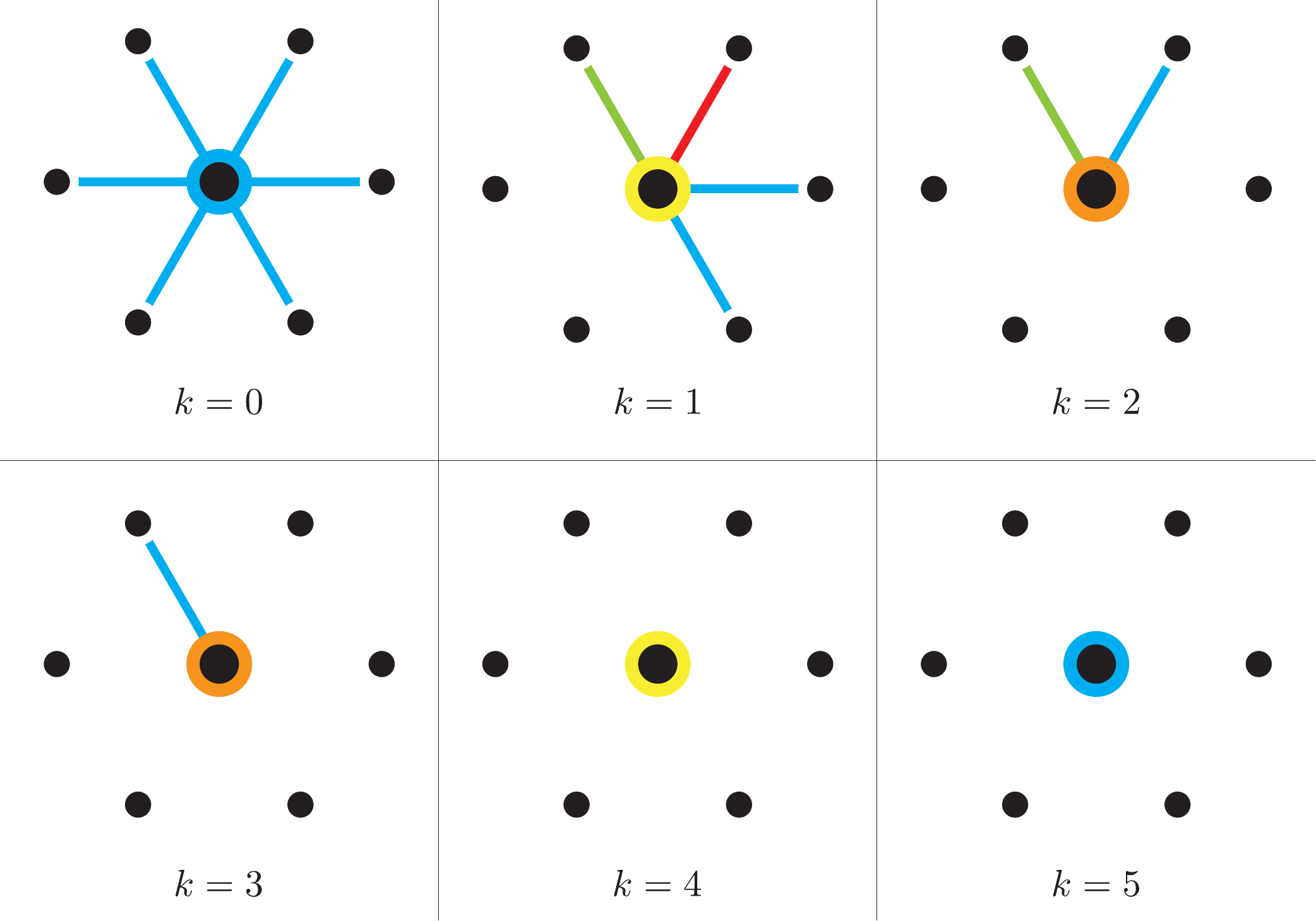}
    \caption{The neighbourhood of the identity of the binomial Cayley graphs on $S_5$, induced by the weight function $\FF$. Colour encircling nodes represents self-loops.}
    \label{fig:sym_cayley_graphs}
\end{figure}

\subsection{Young diagrams, characters of symmetric groups and Crop}\label{section.characters_Sm}

In this section, we recall some facts on the characters of the symmetric group $S_m$ and define $\Crop(\mu, k)$ for a partition $\mu$ of $m$. For a discussion on Young diagrams and characters of symmetric groups, we refer the reader to \cite{Roichman1999}.\\

A \emph{partition} $\mu$ of $m$, denoted by $\mu \vdash m$, is an $\ell$-uple $(\mu_1, \dots, \mu_\ell)$ formed by positive integers satisfying $\mu_1 \ge \dots \ge \mu_\ell$ and $\mu_1 + \dots \mu_\ell = m$. Given a {partition} $\mu \vdash m$, we can consider its {associated} \emph{Young diagram}, namely
\begin{equation*}\label{equation.Young_diagram}
\YD(\mu) = \Set{(i, j) \in \Z \times \Z \ |\ 1 \le i \le \ell \text{ and } 1 \le j \le \mu_i}.
\end{equation*}
We can think of $\YD(\mu)$ as a shape made by $\ell$ left-justified rows of squares of length $\mu_1, \dots, \mu_\ell$, respectively. We refer to the partition $\mu = (\mu_1, \dots, \mu_\ell)$ as the \emph{shape} of the Young diagram.\\

The \emph{boundary} of a Young diagram is the set of squares $(i, j)$ for which $(i+1, j+1)$ does not belong to the Young diagram. A \emph{rim hook} $\tau$ is a connected part of the boundary of a Young diagram which can be removed to leave either a proper Young diagram, or the empty Young diagram (see Figure \ref{fig:young}). The \emph{length} $\len(\tau)$ of a rim hook is the number of squares included in the rim hook. The \emph{height} $\hgt(\tau)$ of a rim hook is one less of the number of rows involved in the rim hook (so that $0 \le \hgt(\tau) \le \ell - 1$ holds for $\mu \vdash m$, $\mu = (\mu_1, \dots, \mu_\ell)$).

\begin{figure}[htb]
    \centering
    \includegraphics[scale=0.2]{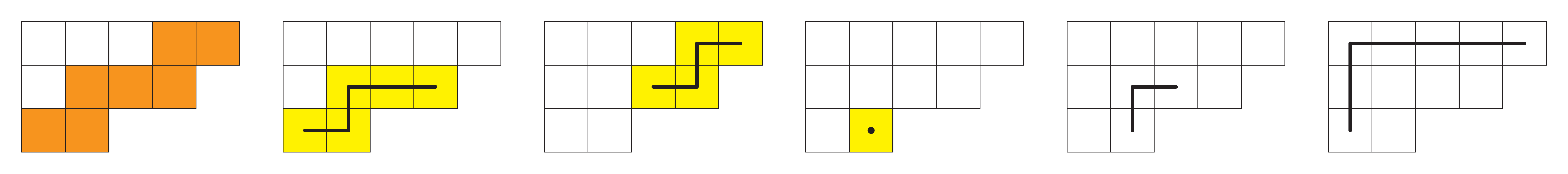}
    \caption{The Young diagram associated with the partition $\mu=(5, 4, 2)$ of $m=11$. From the left: the boundary of the diagram, three rim hooks (of length $5, 4,$ and $1$, and heights $1, 1,$ and $0$, respectively), and two examples of non-rim hooks.}
    \label{fig:young}
\end{figure}

Each partition $\mu$ induces a distinct irreducible character $\chi^\mu$. Moreover, this association is a bijection between the set of partitions of $m$ and the irreducible characters of $S_m$. The value of $\chi^\mu$ on a permutation $\sigma \in S_m$ can be computed using the recursive version of the Murnaghan-Nakayama rule \cite[Section 7.17]{Stanley1999}:
\begin{theorem}\label{theorem.MNrule}
Assume that $\sigma \in S_m$ is the disjoint product $\sigma = \gamma\pi$ of a cycle $\gamma$ permuting $d$ elements and a permutation $\pi$ fixing the same $d$ elements. Then,
\begin{equation*}\label{equation.murnaghan-nakayama}
\chi^\mu(\sigma) = \sum_{\substack{\tau \text{ rim hook of } \YD(\mu) \\ \text{ with } \len(\tau) = d}} (-1)^{\hgt(\tau)} \chi^{\mu \setminus \tau} (\pi).
\end{equation*}
Here $\mu \setminus \tau$ denotes the Young diagram obtained by removing $\tau$ from $\mu$.
\end{theorem}

\begin{definition}
    Given $\mu \vdash m$, and for an integer $k$ satisfying $0 \le k \le m$, we define $\Crop(\mu, k)$ to be the number of ways in which the Young diagram of $\mu$ can be reduced to a single-row Young diagram of length $m-k$ through the progressive removal of $k$ rim hooks of length $1$.
\end{definition}
Notice that the single-row Young diagram of length $m-k$ is the Young diagram of the trivial partition $(m-k)$ of $m-k$.

\begin{example}
    If $m=5$ and $\mu = (2, 2, 1)$, then:
    \begin{enumerate}[label = \roman*)]
        \item $\Crop(\mu, k) = 0$ for $k=0, 1, 2$,
        \item $\Crop(\mu, k) = 2$ for $k=3$,
        \item $\Crop(\mu, k) = 5$ for $k = 4, 5$.
    \end{enumerate}
\end{example}

\begin{lemma}\label{lemma.crop_0}
$\Crop(\mu, k) = 0$ if and only if all the rows of $\YD(\mu)$ have length lower than $m-k$, or equivalently the first row of $\YD(\mu)$ has length lower than $m-k$. Moreover, $\Crop(\mu, k) \le \Crop(\mu, m) = \chi^\mu(\id)$.
\end{lemma}
\begin{proof}
For the first statement, we notice that if $\mu_1 \ge m-k$, there is at least one way to reduce $\YD(\mu)$ to a single-row Young diagram of length $m-k$ (for instance, by always removing the right-most square of the {lowest} row available). Conversely, a single-row Young diagram of length $m-k$ is not reachable if $\mu_1 < m-k$. For the second statement, $\Crop(\mu, k) \le \Crop(\mu, m)$ as every way of reducing $\mu$ to a single-row Young diagram of length $m-k$ extends uniquely to a different way of reducing $\mu$ to the empty Young diagram. Finally, $\Crop(\mu, m) = \chi^\mu(\id)$ follows directly from the iterative application of the Murnaghan-Nakayama rule (Theorem \ref{theorem.MNrule}).
\end{proof}

We also recall the notion of \emph{Young tableau} (see Figure \ref{fig:young_tableau}). A Young tableau is a filling of a Young diagram of shape $\mu \vdash m$ by the numbers $\{ 1, \dots, m \}$, each appearing exactly once. We say that a Young tableau is \emph{standard} if all rows and all columns present numbers in increasing order.

\begin{figure}[htb]
    \centering
    \includegraphics[scale=0.3]{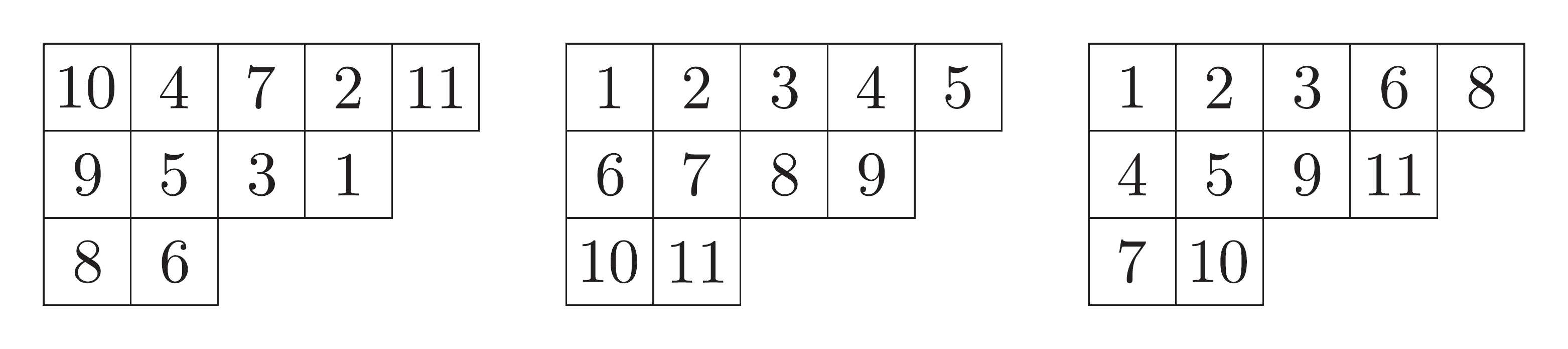}
    \caption{Examples of Young tableaux for the partition $\mu=(5, 4, 2)$ of $m=11$. On the left, a non-standard Young tableau. In the centre and on the right, two standard Young tableaux.}
    \label{fig:young_tableau}
\end{figure}

A consequence of the Murnaghan-Nakayama rule (Theorem \ref{theorem.MNrule}) is the following well-known fact on the number of standard Young tableaux (see for example \cite[Fact 7.6, ii]{Macdonald1995}).
\begin{proposition}\label{proposition.number_standard_young_tableaux}
    Let $\mu \vdash m$ and consider its associated character $\chi^\mu$. Then, the number of standard Young tableaux with shape $\mu$ is $\chi^\mu(\id)$.
\end{proposition}

\begin{remark}\label{remark.Crop_in_literature}
$\Crop(\mu, k)$ coincides in fact with known quantities in the literature. It equals both:
\begin{enumerate}[label = \roman*)]
    \item the number of standard Young tableaux of skew-shape $\mu/(m-k)$ \cite{Sagan1990}, {and}
    \item the Kostka number $K_{\mu, (m-k,1, \dots, 1)}$, where $1$ appears $k$ times \cite[Chapter I, Section 6]{Macdonald1995}.
\end{enumerate}
\end{remark}

\subsection{Spectra of binomial Cayley graphs on symmetric groups}\label{section.spectra_binomial_cayley_graphs_Sm}

In this section we provide a description of the spectrum of the binomial Cayley graph $\Gamma_\Bin(S_m, \FF, k)$.

\begin{lemma}\label{lemma.change_the_binomial}
For every character $\chi$ of $S_m$ and every $f$ such that $k \le f \le m$,
\begin{equation*}\label{equation.swap_binomials_Sm}
\sum_{\substack{\sigma \in S_m \\ \sigma \textrm{ fixes exactly } f \\ \textrm{ elements in } \{1, \dots, m\} }} \chi(\sigma) = 
\sum_{\substack{\sigma \in S_{m - k} \\ \sigma \textrm{ fixes exactly } f-k \\ \textrm{ elements in } \{1, \dots, m-k\} }} \chi(\sigma) \cdot \ \xof{\scb{1.2}{\(\binom{m}{k}\)}}{\scb{1.2}{\(\binom{f}{k}\)}} .
\end{equation*}
Here, $S_{m-k}$ is seen as the subgroup of $S_m$ fixing the $k$ points $\{m-k+1, \dots, m\}$.
\end{lemma}
\begin{proof}
The character $\chi$ is conjugation-invariant. Therefore, for a {given} type of permutation fixing $f$ points in $\{1, \dots, m\}$, the proof reduces to counting how many permutations of {that} type exist in $S_m$, and how many of these belong to $S_{m-k}$. Consider the type
\begin{equation*}\label{equation.type_of_sigma}
(d_{11}, \dots, d_{1t_1},\dots, d_{r1}, \dots, d_{r t_r}, \underbrace{1, \dots, 1}_{f \textrm{ times}}).
\end{equation*}
Here, all the $d_{ij}$ are strictly greater than $1$. Moreover, $d_{ij} = d_i$ for all $j=1, \dots , t_i$, for all $i=1, \dots, r$, and different values of $i$ index different values of $d_i$. The sum of the values $d_{ij}$ is $m-f$ and $1$ is repeated $f$ times. The number of permutations of this type in $S_m$ is
\begin{equation*}
\binom{m}{d_{11}, \dots, d_{rt_r}, \underbrace{1, \dots, 1}_{f \textrm{ times}}} \cdot \frac{(d_{11} - 1)! \cdots (d_{rt_r} - 1)!}{t_1! \cdots t_r! \cdot f!}.
\end{equation*}
Similarly, the number of permutations of this type belonging to $S_{m-k}$ is
\begin{equation*}
\binom{m - k}{d_{11}, \dots, d_{rt_r}, \underbrace{1, \dots, 1}_{f-k \textrm{ times}}} \cdot \frac{(d_{11} - 1)! \cdots (d_{rt_r} - 1)!}{t_1! \cdots t_r! \cdot (f - k)!}.
\end{equation*}
The ratio of these numbers is
\begin{equation*}
\frac{m! \cdot (f-k)!}{f! \cdot (m-k)!} = \xof{\scb{1.2}{\(\binom{m}{k}\)}}{\scb{1.2}{\(\binom{f}{k}\)}} .
\end{equation*}
As this value is independent of the considered permutation type (as soon as it fixes exactly $f$ points in $\{1, \dots, m\}$), the claim follows.
\end{proof}

\begin{lemma}\label{lemma.sum_characters_Sm-k_crop}
Let $\mu \vdash m$ and $k$ such that $0 \le k \le m$. Then,
\begin{equation*}
\sum_{\sigma \in S_{m-k}} \chi^\mu(\sigma) = (m-k)! \cdot \Crop(\mu, k).
\end{equation*}
Here, $S_{m-k}$ is seen as the subgroup of $S_m$ fixing the $k$ points $\{m-k+1, \dots, m\}$.
\end{lemma}
\begin{proof}
We compute $\chi^\mu(\sigma)$ by iterative application of Theorem \ref{theorem.MNrule}. To do so, we apply the Murnagham-Nakayama rule $k$ times, removing one single square from $\YD(\mu)$ at each step. Assume we end up with a partition $\nu$ of $m-k$. If this partition is not the trivial partition $(m-k)$, corresponding to a single-row Young diagram of length $m-k$, then $\sum_{\sigma \in S_{m-k}} \chi^{\nu}(\sigma) = 0$. This follows from the general fact that $\sum_{g \in G} \chi(g) = 0$ for every finite group $G$ and every non-trivial irreducible character $\chi$ of $G$. Therefore, we only need to consider the successive removals of $k$ squares that transform $\YD(\mu)$ into a single-row Young diagram of length $m-k$. By definition, there are $\Crop(\mu, k)$ ways to perform this operation. Finally, if $\nu$ is the trivial partition $(m-k)$ of $m-k$, then $\sum_{\sigma \in S_{m-k}} \chi^{\nu}(\sigma) = (m-k)!$, which proves the claim.
\end{proof}

\begin{theorem}\label{theorem.eigenvalues_binomial_Cayley_graph_Sm}
Let $\mu \vdash m$ and $\chi^\mu$ the associated irreducible character of $S_m$. Then:
\begin{enumerate}
\item The value
\begin{equation*}\label{equation.binomial_Cayley_Sm_graph_eigenvalue}
\lambda_\mu = \frac{\binom{m}{k} \cdot (m-k)! \cdot \Crop(\mu, k)}{\chi^\mu(\id)} = 
\xof[0.2ex]{\scb{.9}{\(\dfrac{\Crop(\mu, k)}{k!}\)\hspace{1pt}}}{\scb{.9}{\hspace{0.5pt}\(\dfrac{\Crop(\mu, m)}{m!}\)}}
\end{equation*}

is an eigenvalue of $\A_\Bin(S_m, \FF, k)$.
\item Its contribute in multiplicity is $\chi^\mu(\id)^2$.
\item $\lambda_\mu$ is an integer number.
\item $\lambda_\mu$ satisfies
\begin{equation}\label{equation.bound_lambda_mu_Sm}
0 \le \lambda_\mu \le \binom{m}{k} \cdot (m-k)!
\end{equation}
\end{enumerate}
\end{theorem}

\begin{remark}\label{remark.inequalities_saturated_Sm}
The latter inequality in Equation \eqref{equation.bound_lambda_mu_Sm} is always saturated. {To see} this, choose the trivial partition $\mu = (m)$, for which $\Crop(\mu, k) = 1 = \chi^\mu(\id)$. On the other hand, the former inequality in Equation \eqref{equation.bound_lambda_mu_Sm} is saturated if and only if $k < m - 1$. To see this, choose the partition $\mu = (1, \dots, 1)$, for which $\Crop(\mu, k) = 0$ unless $k = m-1$ or $k=m$, and conversely $\Crop(\mu, m-1) = \Crop(\mu, m) > 0$ for any partition $\mu$.
\end{remark}

\begin{proof}
By Theorem \ref{theorem.spectra_weighted_normal_Cayley_graphs}, Lemma \ref{lemma.change_the_binomial}, and Lemma \ref{lemma.sum_characters_Sm-k_crop}, we know that the eigenvalue associated with $\mu$ is
\begin{equation*}
\begin{split}
\lambda_\mu &= \frac{1}{\chi^\mu(\id)} \sum_{\sigma \in S_m} \binom{\FF(\sigma)}{k} \cdot \chi^\mu(\sigma) \\
&= \frac{1}{\chi^\mu(\id)} \sum_{f=k}^m \sum_{\substack{\sigma \in S_m \\ \sigma \textrm{ fixes exactly } f \\ \textrm{ elements in } \{1, \dots, m\} }} 
\binom{f}{k} \cdot \chi^\mu(\sigma) \\
&= \frac{1}{\chi^\mu(\id)} \sum_{f=k}^m \sum_{\substack{\sigma \in S_{m-k} \\ \sigma \textrm{ fixes exactly } f-k \\ \textrm{ elements in } \{1, \dots, m-k\} }} \binom{m}{k} \cdot \chi^\mu(\sigma) \\
&= \frac{\binom{m}{k}}{\chi^\mu(\id)} \sum_{\sigma \in S_{m-k}} \chi^\mu(\sigma) \\
&= \frac{\binom{m}{k} \cdot (m-k)! \cdot \Crop(\mu, k)}{\chi^\mu(\id)}.
\end{split}
\end{equation*}
This proves the first equality in (1), while the second equality follows from Lemma \ref{lemma.crop_0}. (2) follows from Theorem \ref{theorem.spectra_weighted_normal_Cayley_graphs}. For (3), it follows from (1) that $\lambda_\mu \in \Q$. Moreover, as $\A_\Bin(S_m, \FF, k)$ has integer entries, $\lambda_\mu$ is an algebraic integer and therefore it belongs to $\Z$. Finally, we prove (4). From (1), it follows that $\lambda_\mu \ge 0$. From lemma \ref{lemma.crop_0}, $\Crop(\mu, k) \le \chi^\mu(\id)$, which implies the other side of the inequality.
\end{proof}

\begin{remark}\label{remark.laplacian_Sm}
An alternative way to prove that $\lambda_\mu \le \binom{m}{k} \cdot (m-k)!$ is through the Laplacian of $\Gamma_\Bin(S_m, \FF, k)$, that is, the matrix
\begin{equation*}
\LL_\Bin(S_m, \FF, k) = d_\Bin(S_m, \FF, k) \cdot \II - \A_\Bin(S_m, \FF, k).
\end{equation*}
As we are assuming that $\omega$ takes non-negative values, $\LL_\Bin(S_m, \FF, k)$ is positive semi-definite and therefore every eigenvalue $\lambda_\mu$ of $\A_\Bin(S_m, \FF, k)$ satisfies $\lambda_\mu \le d_\Bin(S_m, \FF, k)$. On the other hand, following the same argument of Lemma \ref{lemma.change_the_binomial},
\begin{equation*}
\begin{split}
d_\Bin(S_m, \FF, k) &= \sum_{\sigma \in S_m} \binom{\FF(\sigma)}{k} \\
&= \sum_{f=k}^m \binom{f}{k} \cdot \big| \{\sigma \in S_m \text{ fixing exactly $f$ elements in $S_m$}\} \big| \\
&= \sum_{f=k}^m \binom{m}{k} \cdot \big| \{\sigma \in S_{m-k} \text{ fixing exactly $f-k$ elements in $S_{m-k}$}\} \big| \\
&= \binom{m}{k} \cdot (m-k)!
\end{split}
\end{equation*}
which proves the desired inequality.
\end{remark}

\begin{corollary}\label{corollary.rank_binomial_Cayley_graphs_Sm}
The kernel of $\A_\Bin(S_m, \FF, k)$ has dimension
\begin{equation*}\label{equation.dimension_kernel_Sm}
\sum_{\substack{\mu \vdash m \\ \mu_1 < m-k}} \chi^\mu(\id)^2.
\end{equation*}
Equivalently, the rank of $\A_\Bin(S_m, \FF, k)$ is
\begin{equation*}\label{equation.rank_Sm}
\sum_{\substack{\mu \vdash m \\ \mu_1 \ge m-k}} \chi^\mu(\id)^2.
\end{equation*}
\end{corollary}
\begin{proof}
The proof directly follows from Theorem \ref{theorem.eigenvalues_binomial_Cayley_graph_Sm} and Lemma \ref{lemma.crop_0}.
\end{proof}

We now link Corollary \ref{corollary.rank_binomial_Cayley_graphs_Sm} with combinatorial quantities related to increasing sub-sequences and the celebrated RSK correspondence \cite{Robinson1938,Schensted1961,Knuth1970} (although, technically, we are only using the RS version of the RSK correspondence). The RSK correspondence provides a bijection between permutations $\sigma \in S_m$ and pairs $(P, Q)$ of standard Young tableaux of the same shape on $m$. Given a permutation $\sigma \in S_m$, we say that $\sigma$ admits a $k$-\emph{increasing sub-sequence} if there exist $k$ indices $i_1 < \dots < i_k$ in $\{1, \dots, m\}$ such that $\sigma(i_1) < \dots < \sigma(i_k)$. We denote by $\is(\sigma)$ the length of the longest increasing sub-sequence in $\sigma$. The following result (see \cite[Theorem 3]{Schensted1961}) establishes a relation between the longest increasing sub-sequence in $\sigma$ and its associated standard Young tableaux:

\begin{theorem}\label{theorem.RSK_is}
    Consider $\sigma \in S_m$ and let $(P, Q)$ be the standard Young tableaux associated with $\sigma$ via the RSK correspondence. Denote by $\mu = (\mu_1, \dots, \mu_\ell)$ the shape of $P$ and $Q$. Then $\is(\sigma) = \mu_1$.
\end{theorem}

Considering all the possible pairs of standard Young tableaux $(P, Q)$ of the same shape, it follows from Theorem \ref{theorem.RSK_is} and Proposition \ref{proposition.number_standard_young_tableaux} that:

\begin{corollary}\label{corollary.result_cited_on_increasing_sequences}
For every $t$ {such that} $0 \le t \le m$,
\begin{equation*}
\big| \Set{ \sigma \in S_m | \is(\sigma) = t } \big| = \sum_{\substack{\mu \vdash m \\ \mu_1 = t}} \chi^\mu(\id)^2.
\end{equation*}
\end{corollary}

Finally, from Corollary \ref{corollary.rank_binomial_Cayley_graphs_Sm} and Corollary \ref{corollary.result_cited_on_increasing_sequences}:
\begin{corollary}\label{corollary.increasing_sequences}
The dimension of the kernel of $\A_\Bin(S_m, \FF, k)$ coincides with the number of permutations in $S_m$ with no increasing sub-sequences of length $m-k$. Equivalently, the rank of the matrix $\A_\Bin(S_m, \FF, k)$ coincides with the number of permutations in $S_m$ admitting an $(m-k)$-increasing sub-sequence.
\end{corollary}

\subsection{Binomial weight functions of symmetric groups in terms of representations}\label{section.binomial_weight_function_Sm_representation}
In this section, we express the weight function $\FF_k(\sigma) = \binom{\FF(\sigma)}{k}$ in terms of a specific representation of $S_m$.\\

Let $V$ be the vector space over $\C$ with basis $\{e_1, \dots, e_m\}$. $S_m$ acts on $V$ in a natural way: $\sigma \cdot e_i = e_{\sigma(i)}$. A basis for its $k$-th tensor power $V^{\otimes k}$ is $\{e_{i_1} \otimes \cdots \otimes e_{i_k}\}$, for all possible choices of indices $1 \le i_1, \dots, i_k \le m$. The action of $S_m$ extends to $V^{\otimes k}$ by acting on all indices: $\sigma \cdot e_{i_1} \otimes \cdots \otimes e_{i_k} = e_{\sigma(i_1)} \otimes \cdots \otimes e_{\sigma(i_k)}$. Let $W$ be the subspace of $V^{\otimes k}$ generated by the basis elements corresponding to choices of $i_1, \dots, i_k$ all distinct. Let us denote by $\BB$ this basis of $W$. The subspace $W$ has dimension $\binom{m}{k} \cdot k!$, and is in fact an $S_m$-sub-module of $V^{\otimes k}$.\\

Let $\chi_k$ be the character associated with this representation. Since $\sigma \in S_m$ permutes the elements of $\BB$, $\chi_k(\sigma)$ coincides with the number of elements of $\BB$ that are fixed by $\sigma$, that is,
\begin{equation*}
\chi_k(\sigma) = \binom{\FF(\sigma)}{k} \cdot k! = \FF_k(\sigma) \cdot k!
\end{equation*}
As this holds for any $\sigma \in S_m$,
\begin{equation*}
\FF_k = \frac{1}{k!} \cdot \chi_k.
\end{equation*}
From this description, we can deduce in an alternative way that every eigenvalue $\lambda_\mu$ of $\A_\Bin(S_m, \FF, k)$ is non-negative. Indeed, from Theorem \ref{theorem.spectra_weighted_normal_Cayley_graphs},
\begin{equation*}
\lambda_\mu = \frac{m!}{\chi^\mu(\id)} \langle \FF_k, \overline{\chi^\mu} \rangle = \frac{m!}{\chi^\mu(\id) \cdot k!} \langle \chi_k, \overline{\chi^\mu} \rangle \ge 0,
\end{equation*}
as the inner product of two characters is always non-negative. Furthermore, with this interpretation, we can also deduce that $\lambda_\mu$ is an integer, independently of Theorem \ref{theorem.eigenvalues_binomial_Cayley_graph_Sm}. Indeed, $\lambda_\mu \in \Q$ as the inner product of two characters is always an integer number, and the rest follows as before.

\section{Binomial Cayley graphs on powers of cyclic groups}\label{section.binomial_Cayley_graphs_on_Zm^n}
Let $\Z_m$ be the cyclic group of $m$ elements $\{0, \dots, m-1\}$. For a positive integer $n$, we consider its $n$-th power $(\Z_m)^n$ and the weight function $\ZZ \colon (\Z_m)^n \to \N$, {which gives} number of zero coordinates of an $n$-dimensional vector (see Figure \ref{fig:cyclic_cayley_graph_1}). This induces, for every $k \in \N$, the weight function {$\ZZ_k \colon (\Z_m)^n \to \N$ defined as}
\begin{equation*}\label{eq:cyclic_group_weight}
{\ZZ_k(\x) = \binom{\ZZ(\x)}{k}.}
\end{equation*}
Clearly, $\ZZ(-\x) = \ZZ(\x)$ and, as $(\Z_m)^n$ is abelian, $\ZZ$ is a class function. Therefore, we obtain a family of associated weighted normal binomial Cayley graphs $\Gamma_\Bin((\Z_m)^n, \ZZ, k)$ (see Figure \ref{fig:cyclic_cayley_graphs}).

\begin{figure}[htb]
    \centering
    \includegraphics[scale=0.45]{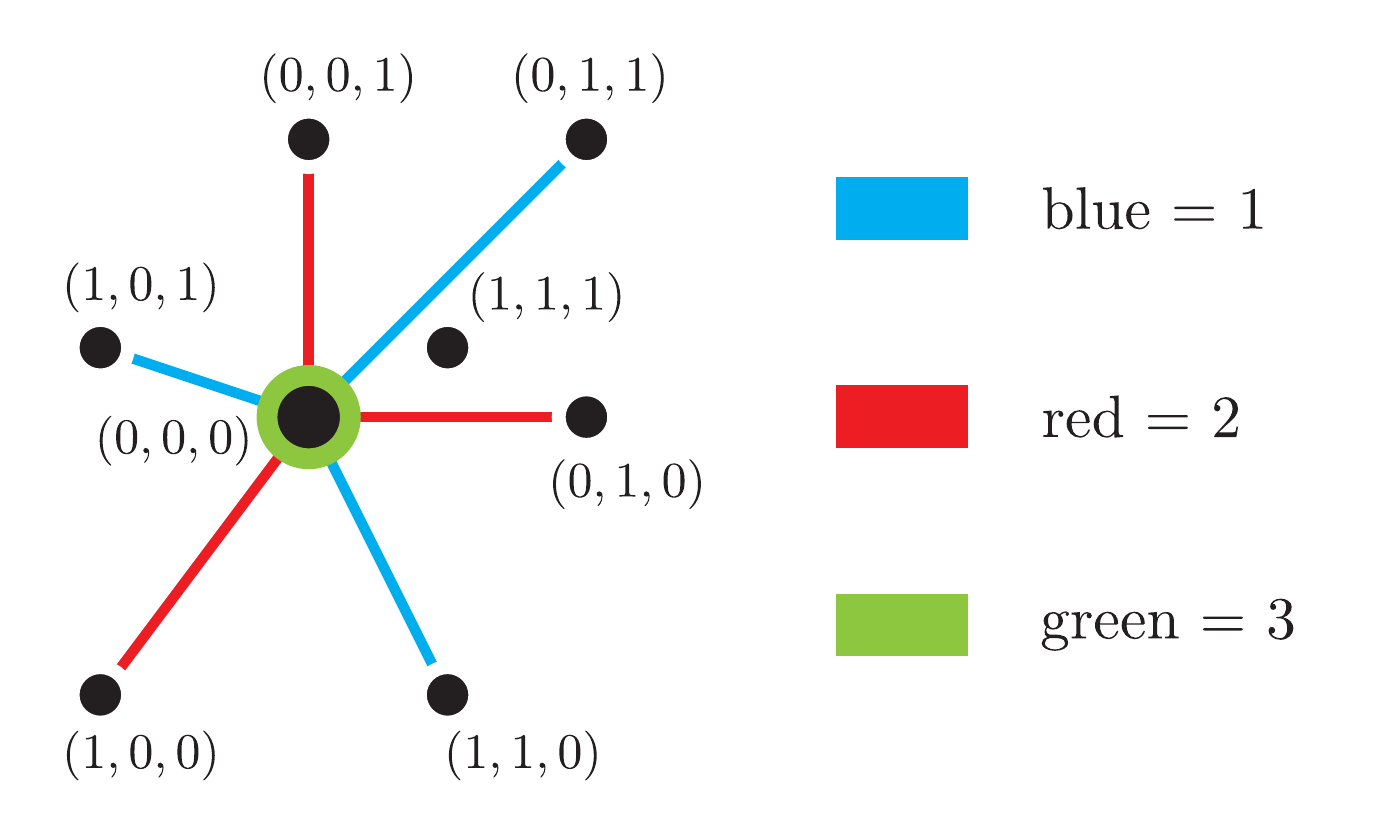}
    \caption{The weighted Cayley graph on the hypercube group $(\Z_2)^3$, induced by the weight function $\ZZ$. Colour encircling nodes represents self-loops.}
    \label{fig:cyclic_cayley_graph_1}
\end{figure}

\begin{figure}[htb]
    \centering
    \includegraphics[scale=0.45]{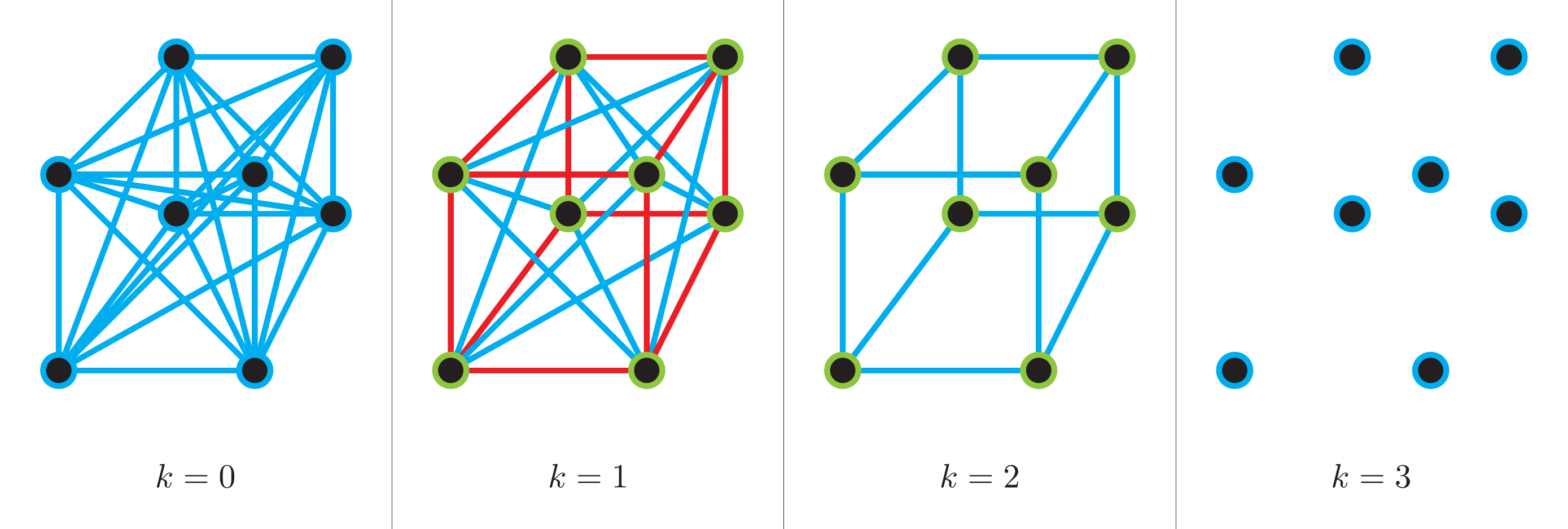}
    \caption{The binomial Cayley graphs on $(\Z_2)^3$, induced by the weight function $\ZZ$. Colour encircling nodes represents self-loops.}
    \label{fig:cyclic_cayley_graphs}
\end{figure}

\begin{remark}
    The discussion carried out in this chapter does not strictly require the group $G$ to be a power of a cyclic group. All consequent results can be extended to groups $G = G_1 \times \dots \times G_n$, where the factors $G_i$ are groups of the same order $m$. For clarity of notation and exposition, we set $G=(\Z_m)^n$.
\end{remark}

\subsection{Characters of powers of cyclic groups}\label{section.characters_Zm^n}
{Given} an abelian group $G$, its \emph{dual group} $G^\dag = \Hom(G, \C^*)$ consists of the irreducible characters of $G$. {Let us} fix a primitive $m$-th root of unity $\zeta \in \C$. For the cyclic group $\Z_m$, this choice induces an isomorphism of groups between $\Z_m$ and its dual group $(\Z_m)^\dag = \Hom(\Z_m, \C^*)$, mapping $y \in \Z_m$ to the irreducible character $\chi^y$, defined for $x \in \Z_m$ as
\begin{equation*}\label{equation.irr_chars_Zm}
\chi^y(x) = \zeta^{y \cdot x}.
\end{equation*}

\begin{lemma}\label{lemma.sum_chi_only_one_Zm^n}
Let $y \in \Z_m$ and {denote by} $\chi^y$ its associated character. Then,
\begin{equation*}
\sum_{x \in \Z_m \setminus \{0\}} \chi^y(x) =
\begin{cases}
m-1 & \text{if $y=0$,} \\
-1 & \text{if $y\neq0$.}
\end{cases}
\end{equation*}
\end{lemma}
\begin{proof}
This follows from properties of roots of unity or from the theory of irreducible characters.
\end{proof}

As irreducible characters of product of groups are products of irreducible characters of their factors, the isomorphism above extends to an isomorphism between $(\Z_m)^n$ and $((\Z_m)^n)^\dag \simeq ((\Z_m)^\dag)^n$, sending $\y \in (\Z_m)^n$ to the irreducible character $\chi^\y$, defined for $\x \in (\Z_m)^n$ as
\begin{equation*}\label{equation.irr_chars_Zm^n}
\chi^\y(\x) = \prod_{i=1}^{n} \zeta^{y_i \cdot x_i} = \zeta^{\sum_{i=1}^n y_i \cdot x_i} = \zeta^{\y \cdot \x}.
\end{equation*}

\subsection{Spectra of binomial Cayley graphs on powers of cyclic groups}\label{section.spectra_binomial_cayley_graphs_Zm^n}

In this section, we provide a description of the spectrum of the binomial Cayley graph $\Gamma_\Bin((\Z_m)^n, \ZZ, k)$.
\begin{lemma}\label{lemma.generating_functions}
Let $M, N, k$ be non-negative integers such that $M+k \ge N$. Then,
\begin{equation}\label{equation.binomial_generating_equation}
\sum_{p=0}^M (-1)^p \cdot \binom{N}{p} \cdot \binom{k + M - p}{M-p} = \binom{k - N + M}{M}.
\end{equation}
\end{lemma}
\begin{proof}
This proof relies on generating functions. By defining
\begin{equation*}
a_p = (-1)^p \cdot \binom{N}{p}, \quad b_p = \binom{k + p}{p},
\end{equation*}
we see that the left-hand side of Equation \eqref{equation.binomial_generating_equation} is the Cauchy product of the sequences $a_p$ and $b_p$. In other words, if we set
\begin{equation*}
A(X) = \sum_{p\ge0} a_p X^p, \quad B(X) = \sum_{p\ge0} b_p X^p,
\end{equation*}
then the left-hand side is the $M$-th coefficient of $C(X) = A(X) \cdot B(X)$. The functions $A(X)$ and $B(X)$ are well-known expansions of
\begin{equation*}
A(X) = (1-X)^N, \quad B(X) = (1-X)^{-(k+1)}.
\end{equation*}
Thus, $C(X) = (1-X)^{N-k-1}$. We distinguish two cases:
\begin{enumerate}[label = \roman*)]
    \item Assume first that $N-k-1 \ge 0$. Then, the $M$-th term in the expansion of $C(X)$ is $(-1)^{M} \cdot \binom{N-k-1}{M} = 0$, as $N-k-1 < N-k \le M$ by hypothesis. On the other hand, $M + k - N < M + k - N + 1 \le M$, so that the right-hand side of Equation \eqref{equation.binomial_generating_equation} is $0$ as well.
    \item Suppose instead that $N-k-1 < 0$. Then, the $M$-th term of the expansion of $C(X) = (1-x)^{-(k-N+1)}$ is $\binom{k - N + M}{M}$, and the proof is concluded. \qedhere
\end{enumerate}
\end{proof}

\begin{lemma}\label{lemma.chi_y_sum_x_m-1_-1}
Let $\y \in (\Z_m)^n$ and consider its associated irreducible character $\chi^\y$ of $(\Z_m)^n$. Also, let $t=\ZZ(\y)$ and $z$ such that $0 \le z \le n$. Then,
\begin{equation*}
\sum_{\substack{\x \in (\Z_m)^n \\ \ZZ(\x) = z}} \chi^\y(\x) = \sum_{\ell=0}^{n-z} (m-1)^\ell \cdot (-1)^{n-z-\ell} \cdot \binom{t}{\ell} \cdot \binom{n-t}{n-z-\ell}.
\end{equation*}
\end{lemma}
\begin{proof}
We see the character $\chi^\y$ as the product of characters $\chi^{y_i}$ on $\Z_m$, for $i=1, \dots, n$. Then, by Lemma \ref{lemma.sum_chi_only_one_Zm^n},
\begin{equation*}
\begin{split}
\sum_{\substack{\x \in (\Z_m)^n \\ \ZZ(\x) = z}} \chi^\y(\x) &= \sum_{\substack{Z \subseteq \{1, \dots, n\} \\ |Z| = z} }\sum_{\substack{\x \in (\Z_m)^n \ \text{s.t.} \\ \text{the set of zeros of $\x$ is $Z$}}} \chi^\y(\x) \\
&= \sum_{\ell=0}^{n-z} \sum_{\substack{Z \subseteq \{1, \dots, n\}, \ |Z| = z \ \text{s.t.} \\ \text{the complement of $Z$ contains} \\ \text{$\ell$ elements of the $t$ zeros of $\y$}}} \quad \sum_{\substack{\x \in (\Z_m)^n \ \text{s.t.} \\ \text{the set of zeros of $\x$ is $Z$}}} \chi^\y(\x) \\
&=\sum_{\ell=0}^{n-z} \sum_{\substack{Z \subseteq \{1, \dots, n\}, \ |Z| = z \ \text{s.t.} \\ \text{the complement of $Z$ contains} \\ \text{$\ell$ elements of the $t$ zeros of $\y$}}} (m-1)^\ell \cdot (-1)^{n-z-\ell} \\
&=\sum_{\ell=0}^{n-z} (m-1)^\ell \cdot (-1)^{n-z-\ell} \cdot {\footnotesize \left| \ \left\{
\begin{array}{c}
Z \subseteq \{1, \dots, n\}, \ |Z| = z \ \text{s.t.} \\
\text{the complement of $Z$ contains} \\
\text{$\ell$ elements of the $t$ zeros of $\y$}
\end{array}
\right\} \ \right|}.
\end{split}
\end{equation*}
Now, a choice for $Z$ as described above is equivalent to a choice of its complement. This is a set of $n-z$ elements: $\ell$ of these belong to the subset $\ZZ(\y)$, of cardinality $t$; the remaining $n - z - \ell$ belong to the complement of $\ZZ(\y)$, of cardinality $n - t$. The number of such choices is $\binom{t}{\ell} \cdot \binom{n-t}{n-z-\ell}$, and the claim is proven.
\end{proof}

\begin{theorem}\label{theorem.eigenvalues_binomial_Cayley_graph_Zm^n}
Let $\y$ be an element of $(\Z_m)^n$ and {denote by} $\chi^\y$ its associated irreducible character of $(\Z_m)^n$. Then:
\begin{enumerate}
\item The value
\begin{equation*}\label{equation.binomial_Cayley_Zm^n_graph_eigenvalue}
\lambda_\y = m^{n-k} \cdot \binom{\ZZ(\y)}{n-k}
\end{equation*}
is an eigenvalue of $\A_\Bin((\Z_m)^n, \ZZ, k)$. \item Its contribute in multiplicity is $1$.
\item $\lambda_\y$ is an integer number.
\item $\lambda_\y$ satisfies
\begin{equation}\label{equation.bound_lambda_mu_Zm^n}
0 \le \lambda_\y \le m^{n-k} \cdot \binom{n}{k}.
\end{equation}
\end{enumerate}
\end{theorem}

\begin{remark}\label{remark.inequalities_saturated_Zm^n}
The latter inequality in Equation \eqref{equation.bound_lambda_mu_Zm^n} is always saturated. {To see} this, choose $\y = \mathbf{0}$. On the other hand, the former inequality in Equation \eqref{equation.bound_lambda_mu_Zm^n} is saturated if and only if $k < n$. To see this, choose $\y = \textbf{1}$, for which $\binom{\ZZ(\y)}{n-k} = 0$ if $k < n$, and conversely $\lambda_\y = 1$ for all $\y$ if $k=n$.
\end{remark}

\begin{proof}
The theorem is stated in the same form as of Theorem \ref{theorem.eigenvalues_binomial_Cayley_graph_Sm}, although (3) and (4) trivially follow from (1) in this case. Again, (2) follows from Theorem \ref{theorem.spectra_weighted_normal_Cayley_graphs}, and the fact that all irreducible characters $\chi$ of abelian groups have degree $\chi(\id)$ equal to $1$. It remains to prove (1). Following Theorem \ref{theorem.spectra_weighted_normal_Cayley_graphs} and Lemma \ref{lemma.chi_y_sum_x_m-1_-1}, writing $t=\ZZ(\y)$, we know the eigenvalue associated with $\y$ is
\begin{equation*}
\begin{split}
\lambda_\y &= \sum_{\x \in (\Z_m)^n} \binom{\ZZ(\x)}{k} \cdot \chi^\y(\x) \\
&= \sum_{z=k}^n \ \binom{z}{k} \cdot \sum_{\substack{\x \in (\Z_m)^n \\ \ZZ(\x) = z}}
\chi^\y(\x) \\
&= \sum_{z=k}^n \binom{z}{k} \cdot \sum_{\ell=0}^{n-z} (m-1)^\ell \cdot (-1)^{n-z-\ell} \cdot \binom{t}{\ell} \cdot \binom{n-t}{n-z-\ell} \\
&= \sum_{\ell=0}^{n-k} (m-1)^\ell \cdot \binom{t}{\ell} \cdot \sum_{z=k}^{n-\ell} (-1)^{n-z-\ell} \cdot \binom{n-t}{n-z-\ell} \cdot \binom{z}{k}.
\end{split}
\end{equation*}
For a fixed $\ell$, we set $p=n-z-\ell$, $M=n-\ell-k$, $N=n-t$, and $k=k$. If $M+k \ge N$, that is, $\ell \le t$, we can apply Lemma \ref{lemma.generating_functions} and rewrite the inner sum as $\binom{t-\ell}{n-\ell-k}$. If instead $M+k < N$, that is, $t < \ell$, the inner sum is multiplied by $\binom{t}{\ell}=0$, so that we can rewrite is as $\binom{t-\ell}{n-\ell-k}$ also in this case. We can then continue the chain of equalities as
\begin{equation*}
\begin{split}
\lambda_\y &= \sum_{\ell=0}^{n-k} (m-1)^\ell \cdot \binom{t}{\ell} \cdot \binom{t-\ell}{n-\ell-k}\\
&= \sum_{\ell=0}^{n-k} (m-1)^\ell \cdot \binom{n-k}{\ell} \cdot \binom{t}{n-k} \\
&= m^{n-k} \cdot \binom{t}{n-k}.
\end{split}
\end{equation*}
As $t=\ZZ(\y)$, the thesis follows.
\end{proof}

\begin{remark}
    {Any time} a binomial with negative entries appears in the argument of the above proof, it is effectively multiplied by a factor $0$. We thus avoid specifying the definition of binomial coefficients with negative entries.
\end{remark}

\begin{remark}\label{remark.laplacian_Zm^n}
As one may expect, {we} can prove the inequality $\lambda_\y \le m^{n-k} \cdot \binom{n}{k}$, also in this context, without relying on the explicit form of the eigenvalues $\lambda_\y$. As done in Remark \ref{remark.laplacian_Sm}, we consider the Laplacian matrix
\begin{equation*}
\LL_\Bin((\Z_m)^n, \ZZ, k) = d_\Bin((\Z_m)^n, \ZZ, k) \cdot \II - \A_\Bin((\Z_m)^n, \ZZ, k).
\end{equation*}
An eigenvalue $\lambda_\y$ of $\A_\Bin((\Z_m)^n, \ZZ, k)$ satisfies $\lambda_\y \le d_\Bin((\Z_m)^n, \ZZ, k)$, as the Laplacian matrix is positive semi-definite (since we assume that $\omega$ takes non-negative values). Moreover, setting $r=n-z$,
\begin{equation*}
\begin{split}
d_\Bin((\Z_m)^n, \ZZ, k) &= \sum_{\x \in (\Z_m)^n} \binom{\ZZ(\x)}{k} \\
&= \sum_{z=k}^n \binom{z}{k} \cdot \big| \{ \x \in (\Z_m)^n \text{ such that } \ZZ(\x) = z \} \big| \\
&= \sum_{z=k}^n \binom{z}{k} \cdot \binom{n}{z} \cdot (m-1)^{n-z} \\
&= \sum_{r=0}^{n-k} \binom{n-r}{k} \cdot \binom{n}{r} \cdot (m-1)^r \\
&= \sum_{r=0}^{n-k} \binom{n}{k} \cdot \binom{n-k}{r} \cdot (m-1)^r \\
&= m^{n-k} \cdot \binom{n}{k}.
\end{split}
\end{equation*}
\end{remark}

\begin{corollary}\label{corollary.rank_binomial_Cayley_graphs_Zm^n}
The spectrum of $\A_\Bin((\Z_m)^n, \ZZ, k)$ is formed by:
\begin{enumerate}[label = \roman*)]
\item The eigenvalue $0$, with multiplicity $\sum_{t<n-k} \binom{n}{t} \cdot (m-1)^{n-t}$.
\item For each $t \in \{n-k, n-k+1, \dots, n\}$, the eigenvalue $m^{n-k} \cdot \binom{t}{n-k}$, with multiplicity $\binom{n}{t} \cdot (m-1)^{n-t}$.
\end{enumerate}
In particular, the kernel of $\A_\Bin((\Z_m)^n, \ZZ, k)$ has dimension
\begin{equation*}\label{equation.dimension_kernel_Zm^n}
\sum_{t<n-k}\binom{n}{t} \cdot (m-1)^{n-t}.
\end{equation*}
Equivalently, the rank of $\A_\Bin((\Z_m)^n, \ZZ, k)$ is
\begin{equation*}\label{equation.rank_Zm^n}
\sum_{s \le k} \binom{n}{s} \cdot (m-1)^s.
\end{equation*}
\end{corollary}
\begin{proof}
We remark that the number of elements $\y \in (\Z_m)^n$ for which $\ZZ(\y) = t$ is $\binom{n}{t}\cdot (m-1)^{n-t}$. Indeed, the set of $t$ zeros of $\y$ can be chosen in $\binom{n}{t}$ ways and the other $n-t$ coordinates can take $m-1$ values each. Then the result follows from Theorem \ref{theorem.eigenvalues_binomial_Cayley_graph_Zm^n}.
\end{proof}

\subsection{Binomial weight functions of powers of cyclic groups in terms of representations}\label{section.binomial_weight_function_Zm^n_representation}
As done in Section \ref{section.binomial_weight_function_Sm_representation}, we express the weight function $\ZZ_k(\sigma) = \binom{\ZZ(\sigma)}{k}$ in terms of a specific representation of $(\Z_m)^n$.\\

Let $V$ be the vector space over $\C$ with basis $\{e_0, \dots, e_{m-1}\}$. We read the indices $\{0, \dots, m-1\}$ as elements of $\Z_m$. The group $\Z_m$ acts on $V$ by $x \cdot e_j$ = $e_{j+x}$. We consider the $n$-th direct power of $V$, $V^n$, with basis $\{e_{i,j}\}$, for the indices $i \in \{1, \dots, n\}$, $j \in \{0, \dots, m-1\}$. The action of $\Z_m$ on $V$ extends to an action of $(\Z_m)^n$ on $V^n$: $\x \cdot e_{i,j} = e_{i,j+x_i}$. We now take the $k$-th tensor power of $V^n$, $(V^n)^{\otimes k}$, with basis
$\{e_{i_1, j_1} \otimes \cdots \otimes e_{i_k, j_k}\}$, for all possible indices $i_1, \dots, i_k \in \{1, \dots, n\}$, and $j_1, \dots, j_k \in \{0, \dots, m\}$. The action of $(\Z_m)^n$ on $V^n$ extends to an action of the same $(\Z_m)^n$ on $(V^n)^{\otimes k}$, by
\begin{equation*}
\x \cdot e_{i_1, j_1} \otimes \cdots \otimes e_{i_k, j_k} = e_{i_1, j_1 + x_{i_1}} \otimes \cdots \otimes e_{i_k, j_k + x_{i_k}}.
\end{equation*}
Let $W$ be the sub-space of $(V^n)^{\otimes k}$ with basis $\BB$ given by the elements $e_{i_1, j_1} \otimes \cdots \otimes e_{i_k, j_k}$ for which $1 \le i_1 < \cdots < i_k \le n$. $W$ has dimension $m^{k} \cdot \binom{n}{k}$ and is in fact a $(\Z_m)^n$-sub-module of $(V^n)^{\otimes k}$.\\

Let $\chi_k$ be the character associated with this representation. Since $\x \in (\Z_m)^n$ permutes the elements of $\BB$, $\chi_k(\x)$ coincides with the number of elements of $\BB$ that are fixed by $\x$, that is,
\begin{equation*}
\chi_k(\x) = \binom{\ZZ(\x)}{k} \cdot m^k = \ZZ_k(\x) \cdot m^k.
\end{equation*}
As this holds for any $\x \in (\Z_m)^n$,
\begin{equation*}
\ZZ_k = \frac{1}{m^k} \cdot \chi_k.
\end{equation*}
From this description and with the same argument used in Section \ref{section.binomial_weight_function_Sm_representation}, we can deduce that every eigenvalue $\lambda_\y$ of $\A_\Bin((\Z_m)^n, \ZZ, k)$ is a non-negative integer number, independently of Theorem \ref{theorem.eigenvalues_binomial_Cayley_graph_Zm^n}.

\section{Degeneracies in particle-box systems}\label{section.degeneracies}

Let us elaborate on the game presented in 
Section \ref{section.introduction}. Consider a set of $n$ particles $\calN = \{0, \dots, n-1\}$ and let $\calM = \{0,\dots, m-1\}$ be a set of $m$ boxes. Denote by $\Sigma \subseteq \{f\colon \calN \to \calM\}$ the so-called set of admissible functions (or arrangements). Let $\Delta$ be the space of probability measures on $\Sigma$,
\begin{equation*}
    \Delta = \Set{ p \in \mathscr{F}(\Sigma, \R) | p(f) \ge 0 \text{ for all } f \in \Sigma, \text{ and } \sum_{f \in \Sigma} p(f) = 1 },
\end{equation*}
where $\mathscr{F}(\Sigma, \R)$ is the space of functions from $\Sigma$ to $\R$. The set of $n$ particles is recursively placed in the $m$ boxes, in a position $(f(1), \dots, f(n)) \in \calM^n$, where each arrangement $f$ is chosen with probability $p(f)$ on every step. Let $k \le n$ be fixed and consider $k$ distinct ordered particles $i_1 < \dots < i_k$ in $\{1, \dots, n\}$ such that
\begin{equation*}
    f(i_1, \dots, i_k) = (f(i_1), \dots, f(i_k)) \in \M^k.
\end{equation*}
We denote by $\ii$ a choice of such $i_1, \dots, i_k$ and write $\ii \in \binom{\mathcal{N}}{k}$. We denote by $\jj = (j_1, \dots, j_k)$ an element in $\calM^k$. Notice that $\jj$ does not need to be increasing. We can always restrict ourselves to considering only admissible sequences $\jj$, i.e. {sequences for which} there exists $f \in \Sigma$ with $f(\ii) = \jj$ for some $\ii$. A multi-index $\ii \in \binom{\calN}{k}$ identifies a marginal probability distribution $p_{(\ii, \cdot)}$ {defined} as
\begin{equation}\label{eq:restriction_k}
    p_{(\ii, \jj)} = \sum_{\substack{f \in \Sigma \\f(\ii) = \jj}} p(f).
\end{equation}
This denotes the probability that each particle $i_1, \dots, i_k$ is placed in box $j_1, \dots, j_k$, respectively and simultaneously, as the game unfolds. We refer to the $\ii$-indexed family of marginal probability measures satisfying Equation \eqref{eq:restriction_k} as the \emph{$k$-th restriction of $p$} and denote it by $p_{|k}$. Altogether, the tuple $(k,n,m,\Sigma, p)$ defines a \emph{particle-box system}.\\

Natural choices of $\Sigma$ include
\begin{equation*}
    \Sigma_{\text{all}} = \{ \ f\colon \calN \to \calM \ \},
\end{equation*}
and if $n = m$
\begin{equation*}
    \Sigma_{\text{bij}} = \{ \ f\colon \calN \to \calM, \ f \text{ bijective} \ \} = \{ \text{ permutations of } \calM \ \}.
\end{equation*}

Let us consider $M^k$ to be the matrix representation of the set of equations given by Equation \eqref{eq:restriction_k}. This is a matrix whose elements are indexed by a pair $(\ii, \jj)$ and a map $f \in \Sigma $, such that
\begin{equation}\label{eq:matrix_M}
    M^k_{({\ii},{\jj}), f} = \begin{cases}
    1 &\text{if }  f(i_1, \dots, i_k) = (j_1, \dots, j_k),\\    
    0 &\text{otherwise.}
\end{cases}
\end{equation}

Assume we have full knowledge over a given $k$-restriction $p_{|k}$, i.e. we are only allowed to observe the movement of any $k$ particles. Our goal is to find the underlying probability distributions $p$ satisfying
\begin{equation}\label{eq:degeneracy}
    M^k p = p_{|k}, \quad p \in \Delta.
\end{equation}
Solutions to this problem depend explicitly on the kernel of $M^k$. Indeed, given an original $p^* \in \Delta$ such that $M^k p^* = p_{|k}$, the set
\begin{equation*}\label{eq:set_of_solutions}
(p^* +  \ker{M^k}) \cap \Delta
\end{equation*}
contains all suitable solutions $p$. Uniqueness of $p$ guarantees that the original distribution $p^*$ can be retrieved from the restriction $p_{|k}$ of the particle-box system (allowing the player to win). This is indeed possible in some cases.

\begin{example}\label{ex.identity}
    Let $k = 1$, $n = m$ and $\Sigma = \Sigma_{\mathrm{all}}$ or $\Sigma_{\mathrm{bij}}$. Assume that $p_{|1}$ is given by $p_{ij} = \delta_{ij}$ for any $i,j \in \calM$. Then the system is uniquely determined by $p = \delta_{\id}$, the atomic distribution sitting on the identity $\id \colon \calM \to \calM$ (see Remark \ref{rmk.degeneracy_122} for further details).
\end{example}

If Equation \eqref{eq:degeneracy} does not have a unique solution, we are unable to recover the original probability $p^*$. This gives rise to a so-called degeneracy in the particle-box system.

\begin{definition}
We define the \emph{degeneracy} of $p^*$ to be the topological dimension of the topological space $(p^* + \ker M^k) \cap \Delta$.
\end{definition}

As $(p^* + \ker M^k) \cap \Delta$ is the intersection of an affine subspace of $\mathscr{F}(G, \R)$ and a simplex, its topological dimension is well defined. Notice that the degeneracy of $p^*$ is equal to $0$ if and only if $p^*$ is the unique element in $(p^* + \ker M^k) \cap \Delta$. For a general $\Sigma$, we state the following result regarding the degeneracy of $p^*$.

\begin{theorem}\label{thm.degeneracy-k-hom}
    The degeneracy of $p^*$ is at most $\dim(\ker M^k)$. If $p^*$ is either in the interior or on the codimension-$1$ boundary of $\Delta$, then the degeneracy of $p^*$ is equal to $\dim(\ker M^k)$.
\end{theorem}
\begin{proof}
    The former statement follows from the fact that $(p^* + \ker M^k) \cap \Delta \subseteq p^* + \ker M^k$. For the latter, we show that the kernel of $M^k$ is parallel to the affine hyperplane containing $\Delta$. Denote by $\mathbf{1}$ the vector of ones.
    Then,
    \begin{equation*}
        (\mathbf{1}^T M^k)_f = \sum_{(\ii, \jj) \in \binom{\calN}{k} \times \M^k} \delta(f(\ii) = \jj) = \sum_{\ii \in \binom{\calN}{k}} 1 = \binom{n}{k}
    \end{equation*}
    is a constant vector. It follows that $\ker M^k$ is orthogonal to $\mathbf{1}$. Finally, if $p^*$ belongs to the interior of $\Delta$, locally around $p^*$, $p^*+\ker M^k$ is contained in $\Delta$. If $p^*$ lies instead on the codimension-$1$ boundary of $\Delta$, locally around $p^*$, a half-space of $p^*+\ker M^k$ is contained in $\Delta$. In both cases, the topological dimension of $(p^* + \ker M^k) \cap \Delta$ coincides with $\dim(\ker M^k)$.
\end{proof}

\begin{remark}\label{rmk.degeneracy_122}
    As already pointed out in Example \ref{ex.identity}, the degeneracy of the system depends on the choice of $p^*$. Assume that $k=1$, $n=m=2$, and $\Sigma=\Sigma_{\mathrm{all}}$. Particles are denoted by the numbers $1$ and $2$, while boxes are denoted by the letters $A$ and $B$. The matrix $M^1$ is
\begin{equation*}
\begin{tabular}{c|c}
        & $\tiny{\begin{matrix} AA & AB& BA &BB \end{matrix}}$\\ \hline \vspace{-.25cm} \\
       $\tiny{\begin{matrix}\vspace{.17cm} 1\mapsto A\\\vspace{.17cm}1\mapsto B\\\vspace{.17cm}2\mapsto A\\\vspace{.17cm}2\mapsto B \\\vspace{-.46cm}\end{matrix}}$&
        ${\begin{pmatrix}
            \ &1\ &&1\ &&0\ &&0\ &\\
            \ &0\ &&0\ &&1\ &&1\ &\\
            \ &1\ &&0\ &&1\ &&0\ &\\
            \ &0\ &&1\ &&0\ &&1\ &\\
        \end{pmatrix}}$
        \end{tabular}
\vspace{.15cm}
\end{equation*}
which has a one-dimensional kernel spanned by the vector $v = (1, -1, -1, 1)$. Adding or subtracting $v$ corresponds to increasing or decreasing the correlation between the variables ``location of particle $1$'' and ``location of particle $2$''. We show the $3$-dimensional simplex $\Delta$ in Figure \ref{fig:tetrahedron}. If an atomic distribution $p^*$ is chosen (say $e_1 = (1, 0, 0, 0)$, corresponding to placing both particles in the first box with probability $1$), $p^* + \ker M^1$ intersects $\Delta$ only at $p^*$. If a distribution $p^*$ in the interior of $\Delta$ is chosen instead (say $\mu = (1/4, 1/4, 1/4, 1/4)$, corresponding to the uniform distribution), then its degeneracy is $1$. Notice that $p^* + \ker M^1$ intersects the boundary of $\Delta$ at two points: $\mu_{14} = \mu + v/2 = (1/2, 0, 0, 1/2)$, corresponding to placing both particles in the same box, chosen with probability $1/2$ (perfect correlation), and $\mu_{23} = \mu - v/2 = (0, 1/2, 1/2, 0)$, corresponding to placing the particles in opposite boxes, chosen with probability $1/2$ (perfect anticorrelation).
\end{remark}

\begin{figure}[htb]
    \centering
    \includegraphics[scale=0.45]{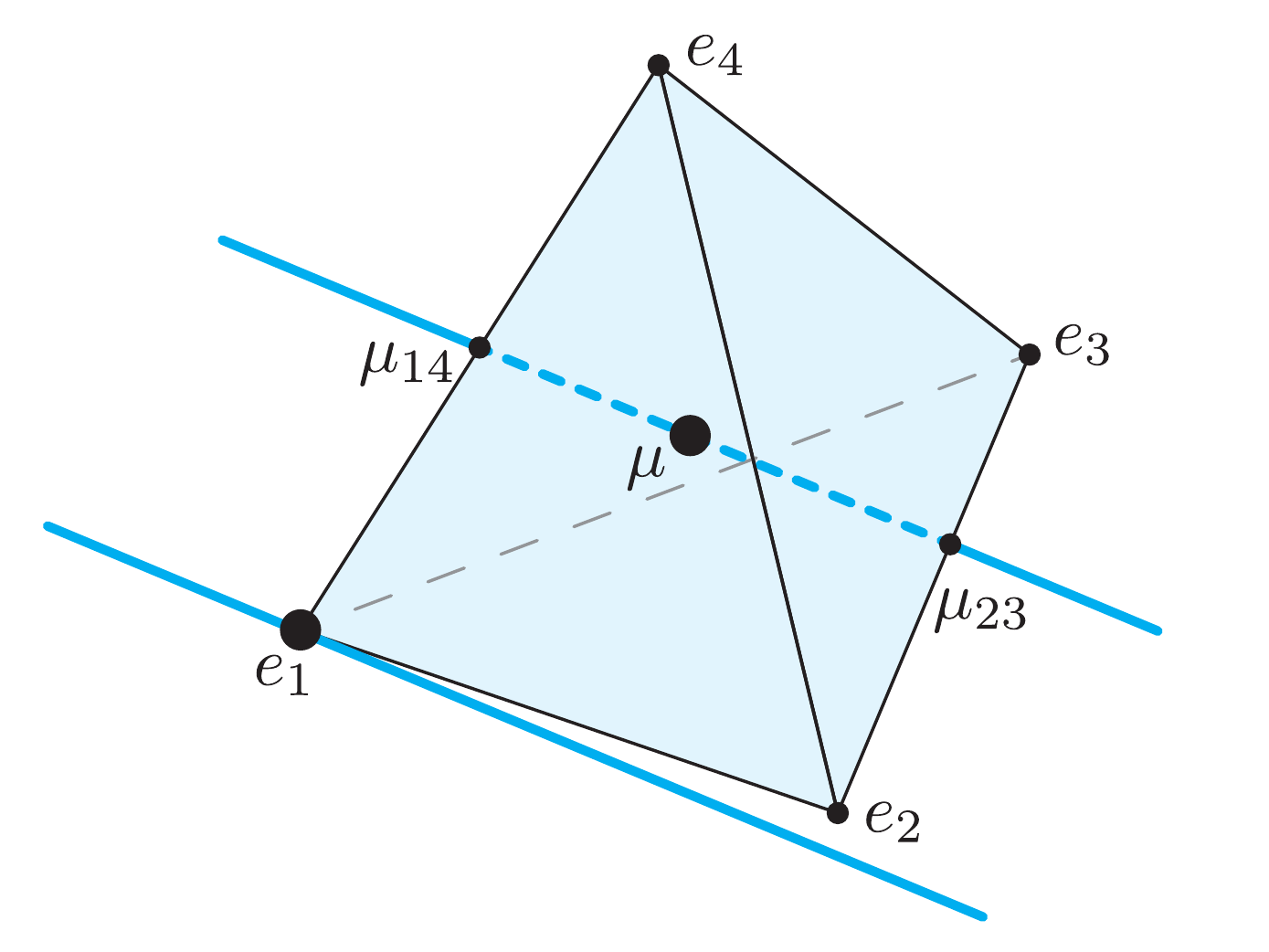}
    \caption{Graphical visualisation of $(p^* + \ker M^k) \cap \Delta$, for $k=1, n=m=2$, $\Sigma=\Sigma_\mathrm{all}$, and $p^* = e_1$ and $p^* = \mu$. In these two cases, the degeneracy of $p^*$ is $0$ and $1$, respectively.}
    \label{fig:tetrahedron}
\end{figure}

As an application, we specialise Theorem \ref{thm.degeneracy-k-hom} to the study of the degeneracy of the particle-box system when $\Sigma = \Sigma_\mathrm{all}$ or $\Sigma_\mathrm{bij}$. It is clear that the kernel of $M^k$ and {the kernel} of the square matrix $A^k = (M^k)^TM^k$ coincide. This simple trick allows us to link this problem to binomial Cayley graphs: a quick computation shows that $A^k$ {is} in fact the adjacency matrix $\mathbf{A}_\Bin((\Z_m)^n, \ZZ, k)$ for $\Sigma_\mathrm{all}$, and $\mathbf{A}_\Bin(S_m, \FF, k)$ for $\Sigma_\mathrm{bij}$.

\subsection{Random maps}\label{sec.random_coalescing_maps}
We begin by considering $\Sigma_{\text{all}}$, allowing any two particles to be placed in the same box. In this setting, we identify the set of states $\calM$ with the cyclic group $\Z_m$. The set of allowed maps $\Sigma_{\text{all}}$ is therefore identified with the $n$-th power of the cyclic group $\Z_m$, i.e. $\Sigma_{\text{all}} \simeq (\Z_m)^n$. \\

It has recently been shown in \cite[Theorem 27]{daCosta2022} that the rank of $M^k$ for each triple $(k, n, m)$ with $1\le k \le n$, denoted by $R^{n,m}_k$, is given by the recursion formula
\begin{equation}\label{eq:recursion}
    R^{n,m}_{k} = R^{n,m}_{k-1} + \binom{n}{k}(m^k - R^{k,m}_{k-1}), \quad R^{n,m}_0 = 1.
\end{equation}

Their proof is based on an inductive argument and the formula follows from explicitly counting the linear restrictions inherited by what the authors refer to as the $(k-1)$-point motion. The initial condition $R^{n,m}_0 = 1$ represents the fact that $p$ is a probability measure (the rank of $M^0 = (1, \dots, 1)$ is $1$). We propose a different approach to answer this question based on the study of the weighted Cayley graph whose adjacency matrix is given by
\begin{equation*}
\begin{split}
    A^k_{fg} &= (M^k)^TM^k\\ & =  \sum_{({\ii}, {\jj}) \in \binom{\calN}{k} \times \calM^k} M^k_{({\ii}, {\jj}), f} \cdot M^k_{({\ii}, {\jj}), g}\\
    & = \left| \ \Set{{\ii} \in \binom{\calN}{k} | f({\ii}) = g ({\ii})} \ \right|\\
     &= \binom{\ZZ( g - f ) }{k},
\end{split}
\end{equation*}
with $f, g \in (\Z_m)^n$. {As mentioned, the} matrix $A^k$ coincides with the adjacency matrix $\A_\Bin((\Z_m)^n, \ZZ, k)$ of the binomial Cayley graph $\Gamma_\Bin((\Z_m)^n, \ZZ, k)$ presented in Section \ref{section.binomial_Cayley_graphs_on_Zm^n}. It follows from Corollary \ref{corollary.rank_binomial_Cayley_graphs_Zm^n} that the rank of $A^k$, and hence the number of restrictions $R^{n,m}_k$, is given by
\begin{equation}\label{eq:R_all}
R^{n,m}_{k} = \sum_{t=0}^k \binom{n}{t} \cdot (m-1)^t,
\end{equation}
which solves the recurrence relation provided in Equation \eqref{eq:recursion}. We provide an example of this setting in Example \ref{ex:compatible-non-observable}. It follows from Theorem \ref{thm.degeneracy-k-hom} and Corollary \ref{corollary.rank_binomial_Cayley_graphs_Zm^n} that:

\begin{corollary}\label{cor.degeneracy_random_maps}
    For $\Sigma=\Sigma_{\mathrm{all}}$, the degeneracy of $p^*$ is at most
    \begin{equation*}
        \sum_{t<n-k} \binom{n}{t} \cdot (m-1)^{n-t}.
    \end{equation*}
    If $p^*$ is either in the interior or on the codimension-$1$ boundary of $\Delta$, then the equality of dimensions holds.
\end{corollary}

\begin{remark}
    In this section, the matrix $M^k$ corresponds to the restriction matrix of a particle-box system from $n$ observed particles to $k \leq n$. When considering $\Sigma = \Sigma_{\mathrm{all}}$, the matrix $M^k$ introduced in this section falls within the category of so-called inclusion matrices for rainbow $k$-subsets \cite{Qian2023}. As pointed out by the authors in \cite[Section 1]{Qian2023}, the matrix $M^k$ can be interpreted as the matrix $\mathcal{W}^{m, \dots, m}_{k, n}$ introduced in their paper, where $m$ is repeated $n$ times. Independently of (binomial) Cayley graphs and character theory, their line of research provides spectral results for a more general family of matrices, denoted by $\mathcal{W}^{a_1, \dots, a_n}_{t, k}$, including an analogous result of Corollary \ref{cor.degeneracy_random_maps}. Moreover, the authors also obtain a formula for the singular values of such matrices in \cite[Theorem 3.47]{Qian2023}, which coincide with the singular values of $M^k$. These singular values are the square root of the eigenvalues of $\A_\Bin((\Z_m)^n, \ZZ, k)$, described in Theorem \ref{theorem.eigenvalues_binomial_Cayley_graph_Zm^n}. Finally, \cite[Theorem 3.50]{Qian2023} provides a basis of eigenvectors  of $\A_\Bin((\Z_m)^n, \ZZ, k)$ consisting only of $\{0, \pm 1\}$ entries, as pointed out by the authors in \cite[Remark 5.7]{Qian2023}. We refer the reader to \cite{Qian2023} for more details on the topic of inclusion matrices and rainbow $k$-subsets.
\end{remark}

\subsection{Random bijections} Let $n = m$ and consider $\Sigma_{\text{bij}}$, corresponding to the set of permutations of $m$ elements $S_m$. Denote by $\calM^k_{\neq}$ the subset of $\calM^k$ formed by tuples with all different entries. Following the arguments provided in the beginning of this section, the $k$-th restriction matrix $M^k$ is given by Equation \eqref{eq:matrix_M}, now parameterised by $\sigma, \tau \in S_m$, such that 
\begin{equation*}
\begin{split}
    A^k_{\sigma \tau} &= (M^k)^TM^k\\ 
    & =  \sum_{({\ii}, {\jj}) \in \binom{\calM}{k} \times \calM^k_{\neq}} M^k_{({\ii}, {\jj}), \sigma} \cdot M^k_{({\ii}, {\jj}), \tau}\\
    & = \left| \ \Set{ {\ii} \in \binom{\calM}{k} | \sigma(\ii) = \tau (\ii)} \ \right|\\
     &= \binom{\FF(\sigma^{-1}\tau) }{k} \\
     &= \binom{\FF(\tau\sigma^{-1}) }{k},
\end{split}
\end{equation*}
where $\FF$ denotes the number of elements fixed by a permutation in $S_m$. The last equality follows from the fact that $\FF$ is a class function. The matrix $A^k$ coincides with the adjacency matrix $\A_\Bin(S_m, \FF, k)$ of the binomial Cayley graph $\Gamma_{\Bin}(S_m, \FF, k)$ presented in Section \eqref{section.binomial_Cayley_graphs_on_Sm}. It follows from Corollary \ref{corollary.increasing_sequences} that the rank of $M^k$ coincides with the number of permutations in $S_m$ admitting an $(m-k)$-increasing sub-sequence, proving the result proposed in \cite[Conjecture 30]{daCosta2022}. It follows from Theorem \ref{thm.degeneracy-k-hom} and Corollary \ref{corollary.increasing_sequences} that:

\begin{corollary}\label{cor.degeneracy_random_bijections}
    For $\Sigma=\Sigma_{\mathrm{bij}}$, the degeneracy of $p^*$ is at most the number of permutations in $S_m$ with no increasing sub-sequences of length $m-k$. If $p^*$ is either in the interior or on the codimension-$1$ boundary of $\Delta$, then the equality of dimensions holds.
\end{corollary}

\section{Final remarks}\label{sec.final}

In this final section, we discuss some remarks on eigenvalues and eigenvectors of binomial Cayley graphs. Moreover, we provide a counterexample to show that a nested-kernel property valid for the binomial Cayley graphs considered in this paper is not valid in general. Finally, we present observable and compatible families of probabilities and an example of compatible non-observable family.

\subsection{Eigenvalues of binomial Cayley graphs: a general formula}\label{sec.ugly_formula_general_eigenvalues_binomial} In this paper, we studied the spectra of specific families of binomial Cayley graphs on symmetric groups and on powers of cyclic groups (Theorems \ref{theorem.eigenvalues_binomial_Cayley_graph_Sm} and \ref{theorem.eigenvalues_binomial_Cayley_graph_Zm^n}). Here, we underline that a general formula for the eigenvalues of binomial Cayley graphs $\Gamma_\Bin(G, \omega, k)$ in terms of the eigenvalues of the original weighted normal Cayley graph $\Gamma(G, \omega)$ exists. However, it does not seem to be easily tractable. This motivates the specialised analyses carried out in Sections \ref{section.binomial_Cayley_graphs_on_Sm} and \ref{section.binomial_Cayley_graphs_on_Zm^n}, and the general structure of the paper. \\

Theorem \ref{theorem.spectra_weighted_normal_Cayley_graphs} expresses the eigenvalues $\lambda_\chi$ of the adjacency matrix of a weighted normal Cayley graphs in terms of the scalar products $\langle \omega, \overline{\chi} \rangle$. This can be rephrased from another point of view. Under the normality assumption, $\omega$ is a class function and can then be written as
\begin{equation*}
    \omega = \sum_\chi a_\chi \chi,
\end{equation*}
with the sum running over the irreducible characters $\chi$ of $G$. As these constitute an orthonormal basis of the space of class functions,
\begin{equation*}
    a_\chi = \langle \omega, \chi \rangle = \frac{\chi(\id)}{|G|} \lambda_{\overline{\chi}}.
\end{equation*}
From this, we see that determining the coefficients $a_\chi$ for a class function $\omega$ is equivalent to determining the spectrum of its associated weighted Cayley graph. We express such coefficients for $\omega_k$ in terms of the coefficients of $\omega$. For $i = 0, \dots, k-1$, we define
\begin{equation*}
    a^{-i}_\chi =
    \begin{cases}
        a_\chi & \text{if } \chi \neq \mathbf{1}, \\
        a_\chi - i & \text{if } \chi = \mathbf{1}.
    \end{cases}
\end{equation*}
Here we denote by $\mathbf{1}$ the trivial character $\mathbf{1}(g) = 1$ for all $g \in G$. We recall that
\begin{equation*}
    \binom{\omega(g)}{k} = \frac{\omega(g) \cdot \left(\omega(g) - 1\right) \cdots \left( \omega(g) - k + 1 \right)}{k!}.
\end{equation*}
We can then use the above defined $a^{-i}_\chi$ to determine the decomposition of the factors $\omega - i$ of the binomial coefficient, that is:
\begin{equation*}
    \omega - i = \sum_{\chi} a^{-i}_\chi \chi.
\end{equation*}
Consequently, the coefficients $\alpha_\chi$ of the class function $\omega_k = \sum_\chi \alpha_\chi \chi$ can be expressed as
\begin{equation*}
    \alpha_\chi = \frac{1}{k!} \sum_{\chi_0, \dots, \chi_{k-1}} a^{-0}_{\chi_0} \cdots a^{-(k-1)}_{\chi_{k - 1}} \cdot \langle \chi_0 \cdots \chi_{k-1} , \chi \rangle.
\end{equation*}
This expression depends on the decomposition of the product of irreducible characters as a sum of irreducible characters.

\subsection{Eigenvectors of weighted normal Cayley graphs}\label{sec.eigenvectors_binomial} Let $\mathscr{F}(G, \C)$ be the space of $\C$-valued functions on the group $G$, constituting the nodes of a weighted Cayley graph $\Gamma(G, \omega)$. The \emph{regular representation} $\rho_\mathrm{reg} \colon G \to \mathrm{GL}(\mathscr{F} (G, \C))$ is defined, for $s, g \in G$ and $f \in \mathscr{F}(G, \C)$, as
\begin{equation*}
    (\rho_\mathrm{reg} (s) f) (g) = f(s^{-1} g).
\end{equation*}
Following \cite[Section 1]{Rockmore2002}, we consider the $\C$-linear operator $\mathcal{A} \colon \mathscr{F}(G, \C) \to \mathscr{F}(G, \C)$,
\begin{equation*}
    \mathcal{A} = \sum_{s \in G} \omega(s) \rho_\mathrm{reg} (s).
\end{equation*}
The $\C$-vector space $\mathscr{F}(G, \C)$ has, as $\C$-basis, the set $\{\delta_g\}_{g \in G}$, where
\begin{equation*}
    \delta_g(h) = 
    \begin{cases}
        1 & \text{if } g = h, \\
        0 & \text{if } g \neq h.
    \end{cases}
\end{equation*}
The action of $\mathcal{A}$ on this basis can be computed as
\begin{equation*}
    \mathcal{A}\delta_h (g) = \sum_{s \in G} \omega(s) \delta_h(s^{-1} g) = \omega(gh^{-1}),
\end{equation*}
from which we deduce
\begin{equation*}
    \mathcal{A} \delta_h = \sum_{g \in G} \omega(gh^{-1}) \delta_g.
\end{equation*}
In other words, the transpose of the adjacency matrix $\mathbf{A}(G, \omega)$ of the Cayley graph $\Gamma(G, \omega)$ represents the linear operator $\mathcal{A}$ with respect to the basis $\{\delta_g\}_{g \in G}$ (the transpose appears as $\mathcal{A}$ acts on functions on the graph nodes). As we assume that $\mathbf{A}(G, \omega)$ is symmetric, $\mathbf{A}(G, \omega)$ represents $\mathcal{A}$. When talking about eigenvectors of a Cayley graph, we refer to the eigenfunctions of the linear operator $\mathcal{A}$. For completeness, we mention that a dual alternative approach to identify eigenvectors can be developed for the group algebra $\C[G]$, where the adjacency matrix represents the multiplication by a certain element of $\C[G]$ (see for instance \cite[Section 4]{Roichman1999}). \\

We recall that $\omega$ takes values in $\N$ (in particular, $\omega(g) = \overline{\omega(g)}$ for all $g \in G$), is invariant under inversion and is a class function. Then, for each irreducible character $\chi$ of $G$, Theorem \ref{theorem.spectra_weighted_normal_Cayley_graphs} identifies $\lambda_\chi = \frac{|G|}{\chi(\id)} \langle \omega, \overline{\chi} \rangle$ as an eigenvalue of $\mathbf{A}(G, \omega)$. Its contribute in multiplicity is $\chi(\id)^2$. Let $\rho_\chi \colon G \to \mathrm{GL}_{\chi(\id)}(\C)$ be a matrix irreducible representation -- unique up to isomorphism of representations -- inducing the character $\chi = \mathrm{Tr}(\rho_\chi)$. For each $i, j$ such that $1 \le i, j \le \chi(\id)$, the matrix element $\rho_{\chi, ij} \colon G \to \C$ is an element of $\mathscr{F}(G, \C)$. It turns out that the $\chi(\id)^2$ complex conjugate functions $\overline{\rho_{\chi, ij}}$ are linearly independent eigenfunctions relative to the eigenvalue $\lambda_\chi$ \cite[Theorem 1.1]{Rockmore2002}. We say that these eigenfunctions constitute the contribute of $\chi$ to the eigenspace $V_{\lambda_\chi}$, as different characters might induce the same eigenvalues. In particular, for a given $\lambda$, a basis for the eigenspace $V_{\lambda}$ is given by
\begin{equation*}
    \bigcup_{\chi \colon \lambda_\chi = \lambda}
    \left\{ \ \overline{\rho_{\chi, ij}} \ \right\}_{1 \le i, j \le \chi(\id)}.
\end{equation*}
This description highlights a remarkable property shared by all the weighted normal Cayley graphs on a group $G$: the eigenfunctions above uniquely depend on the group structure $G$ and not on the weight function $\omega$. It follows that, as $k$ varies, each binomial Cayley graph $\Gamma_{\Bin}(G, \omega, k)$ shares the same eigenvectors with the original weighted normal Cayley graph $\Gamma(G, \omega)$. \\

It is worth noticing that the invariance under inversion of $\omega$ implies that $\lambda_{\overline{\chi}} = \lambda_\chi$, and the fact that $\omega$ is real-valued implies that 
$\lambda_{\overline{\chi}} = \overline{\lambda_\chi}$. Combining the two, we deduce that $\chi$ and $\overline{\chi}$ induce the same real eigenvalue. If $\rho_\chi$ induces the character $\chi$, then we might choose $\rho_{\overline{\chi}} = \overline{\rho_\chi}$. Finally, from the discussion above, we see that the eigenspace $\lambda_\chi$ is invariant under complex conjugation, and therefore admits a basis of real valued-functions on $G$, as we expected.

\subsection{Kernels are generally not nested}\label{sec.nested_kernels} In Section \ref{section.binomial_Cayley_graphs_on_Sm} (respectively Section \ref{section.binomial_Cayley_graphs_on_Zm^n}) we derive the expression for the dimension of the kernel of $\A_\Bin(G, \omega, k)$, for $(G, \omega) = (S_m, \FF)$ (respectively $((\Z_m)^n, \ZZ)$). The kernels of these matrices correspond to the kernels of their respective $M^k$, as described in Section \ref{section.degeneracies}. As $k$ increases from $0$ to $m$ (respectively from $0$ to $n$), the dimension of the kernel of $\A_\Bin(G, \omega, k)$ diminishes. Not only this: the kernels are nested, in the sense that
\begin{equation*}
    \ker(M^{k+1}) \subseteq \ker(M^k),
\end{equation*}
for $0 \le k \le m-1$ (respectively $0 \le k \le n-1$). To see this, from the discussion in Section \ref{sec.eigenvectors_binomial} and Corollary \ref{corollary.rank_binomial_Cayley_graphs_Sm}, the kernel of $M^k$ when $(G, \omega) = (S_m, \FF)$ is
\begin{equation*}
    \ker(M^k) = 
    \bigcup_{\substack{\mu \vdash m \\ \mu_1 < m-k}}
    \left\{ \ \overline{\rho_{\chi^\mu, ij}} \ \right\}_{1 \le i, j \le \chi^\mu(\id)}.
\end{equation*}
When instead $(G, \omega) = ((\Z_m)^n, \ZZ)$,
\begin{equation*}
    \ker(M^k) = \bigcup_{
    \substack{\y \in (\Z_m)^n \\ \ZZ(\y) < n-k}}
    \left\{ \ \overline{\rho_{\chi^\y, ij}} \ \right\}_{1 \le i, j \le \chi^\y(\id)}.
\end{equation*}

Alternatively, the inclusion of the kernels follows from the existence of matrices $P_{k+1}^k$ such that $M^k = P_{k+1}^k M^{k+1}$, for $0 \le k \le m-1$ (respectively, $0 \le k \le n-1$). These matrices are indexed by suitable double pairs of indices $(\ii_k, \jj_k), (\ii_{k+1}, \jj_{k+1})$, where $\ii_k, \jj_k$ have length $k$, and $\ii_{k+1}, \jj_{k+1}$ have length $k+1$. The corresponding entry of the matrix $P_{k+1}^k$ has value $C_k$ if $\ii_k \subseteq \ii_{k+1}$ and $\jj_{k+1}$ coincides with $\jj_k$ in correspondence of $\ii_k$, and $0$ otherwise. The value $C_k$ is $1/(m-k)$ for $(G, \omega) = (S_m, \FF)$, and $1/(n-k)$ for $(G, \omega) = ((\Z_m)^n, \ZZ)$.\\

Of course, the nested property is no longer satisfied when $k$ reaches the maximal value of $\omega$, as in this case $M^{k+1}$ -- as well as the adjacency matrix of the associated binomial Cayley graph -- is the zero matrix. We show with an example that, in general, there exist weighted Cayley graphs for which the inclusion $\ker(\A_\Bin(G, \omega, k+1)) \subseteq \ker(\A_\Bin(G, \omega, k))$ fails before $k$ reaches the maximal value of $\omega$.
\begin{example}
Let $G = \Z_4$ and consider the weight function $\omega$ defined by
\begin{equation*}
    \omega(g) =
    \begin{cases}
        4 & \text{if } g = 0, \\
        7 & \text{if } g = \pm 1, \\
        9 & \text{if } g = 2.
    \end{cases}
\end{equation*}
A computation of the determinant of the adjacency matrices of the binomial Cayley graphs $\Gamma_\Bin(G, \omega, k)$ shows that $\A_\Bin(G, \omega, 1) = \A(G, \omega)$ is non-singular, while $\A_\Bin(G, \omega, 2)$ has non-trivial kernel.
\end{example}

\subsection{Observable and compatible families of probabilities}\label{sec.observable-compatible} Assume that $p_{|k}$ is a family of restricted probability measures coming from a probability $p \in \Delta$, that is $p_{|k} = M^k p$ as described in Section \ref{section.degeneracies}. The family $p_{|k}$ satisfies the following compatibility condition when further restricted. For any $s$, $1 \le s \le k$, we denote by $\ii_s \in \binom{\calN}{s}$ and $\jj_s \in \calM^s$ multi-indices with emphasis on their length $s$. If $\ii_s \subseteq \ii_k$, we also denote by ${\jj_k}_{|\ii_s}$ the sub-multi-index of $\jj_k$ of length $s$ corresponding to the indices $\ii_s$. Then, we have that
\begin{equation}\label{eq:compatibility}
    p_{(\ii_s, \jj_s)} = \sum_{\jj_k \colon {\jj_k}_{|\ii_s} = \jj_s} p_{(\ii_k, \jj_k)}
\end{equation}
is independent of the choice of $\ii_k$ containing $\ii_s$. It is natural to ask if a given family of probabilities satisfying this condition coincides in fact with $M^k p$ for some $p \in \Delta$.
\begin{definition}
    We say that a family of probabilities $\{p_{(\ii_k, \cdot)}\}_{\ii_k}$ is \emph{compatible} if the quantity defined in Equation \eqref{eq:compatibility} is independent of $\ii_k$ containing $\ii_s$, for all $\ii_s \in \binom{\calN}{s}$.
\end{definition}
\begin{definition}
    We say that a family of probabilities $\{p_{(\ii_k, \cdot)}\}_{\ii_k}$ is \emph{observable} if there exists $p \in \Delta$ inducing the family through $M^k$, i.e. $M^k p = p_{|k}$.
\end{definition}
It turns out that requiring the family $\{p_{(\ii_k, \cdot)}\}_{\ii_k}$ to be compatible is necessary for it to be observable but not sufficient, as shown in the following example.

\begin{example}\label{ex:compatible-non-observable}
Consider the case $k=2$, $n = 3$ and $m = 2$ with $\Sigma = \Sigma_{\text{all}}$. We denote the particles by the numbers $1, 2$, and $3$, and the boxes by the letters $A$ and $B$. Let $\{p_{(\ii_2, \cdot)}\}_{\ii_2}$ be the family given by
\begin{equation*}
\begin{split}
        p(\ii_2,(A,A)) = \frac{1}{12}, \quad p(\ii_2,(A,B)) = \frac{11}{30},\\
        p(\ii_2,(B,A)) = \frac{11}{30}, \quad p(\ii_2,(B,B)) = \frac{11}{60},
\end{split}
\end{equation*}
where $\ii_2 \in \{(1,2),(1,3),(2,3)\}$. This is a compatible family since for any $\ii_2$ containing element $i=1, 2, 3$ we have
\begin{equation*}
\begin{split}
    p(i,A) &= \frac{1}{12} + \frac{11}{30},\\
    p(i,B) &= \frac{11}{30} +\frac{11}{60}.
\end{split}
\end{equation*}
Furthermore, $M^2$ maps the signed probability measure (i.e. a function in $\mathscr{F}(\Sigma, \R)$)
\begin{equation*}
p^{*} = \left(-\frac{6}{60}, \frac{11}{60}, \frac{11}{60},\frac{11}{60},\frac{11}{60},\frac{11}{60},\frac{11}{60},0 \right)
\end{equation*}
to the family $\{p(\ii_2, \cdot)\}_{\ii_2}$. Following the discussion of Section \ref{section.degeneracies}, we are interested in the kernel of $M^2$ to show there are no probability measures whose image by $M^2$ is this same family. In this case, $M^2$ is
\begin{equation*}
\begin{tabular}{c|c}
        & $\tiny{\begin{matrix} {\hspace{.05cm} AAA\hspace{.05cm} } &{\hspace{.05cm} AAB\hspace{.05cm} }&{\hspace{.05cm} ABA\hspace{.05cm} }&{\hspace{.05cm} ABB\hspace{.05cm} }&{\hspace{.05cm} BAA\hspace{.05cm} } &{\hspace{.05cm} BAB\hspace{.05cm} }&{\hspace{.05cm} BBA\hspace{.05cm} }&{\hspace{.05cm} BBB\hspace{.05cm} } \end{matrix}}$\\ \hline \vspace{-.25cm} \\
       $\tiny{\begin{matrix}
       \vspace{.174cm} 12\mapsto AA\\
       \vspace{.174cm} 12\mapsto AB\\
       \vspace{.174cm} 12\mapsto BA\\
       \vspace{.174cm} 12\mapsto BB\\
       \vspace{.174cm} 13\mapsto AA\\
       \vspace{.174cm} 13\mapsto AB\\
       \vspace{.174cm} 13\mapsto BA\\
       \vspace{.174cm} 13\mapsto BB\\
       \vspace{.174cm} 23\mapsto AA\\
       \vspace{.174cm} 23\mapsto AB\\
       \vspace{.174cm} 23\mapsto BA\\
       \vspace{.174cm} 23\mapsto BB\\
       \vspace{-.46cm}\end{matrix}}$&
    ${\begin{pmatrix}
    \hspace{.28cm}1\hspace{.28cm} &
    \hspace{.28cm}1\hspace{.28cm} &
    \hspace{.28cm}0\hspace{.28cm} &
    \hspace{.28cm}0\hspace{.28cm} &
    \hspace{.28cm}0\hspace{.28cm} &
    \hspace{.28cm}0\hspace{.28cm} &
    \hspace{.28cm}0\hspace{.28cm} &
    \hspace{.28cm}0\hspace{.28cm}\\

    \hspace{.2cm}0\hspace{.2cm} &
    \hspace{.2cm}0\hspace{.2cm} &
    \hspace{.2cm}1\hspace{.2cm} &
    \hspace{.2cm}1\hspace{.2cm} &
    \hspace{.2cm}0\hspace{.2cm} &
    \hspace{.2cm}0\hspace{.2cm} &
    \hspace{.2cm}0\hspace{.2cm} &
    \hspace{.2cm}0\hspace{.2cm}\\

    \hspace{.2cm}0\hspace{.2cm} &
    \hspace{.2cm}0\hspace{.2cm} &
    \hspace{.2cm}0\hspace{.2cm} &
    \hspace{.2cm}0\hspace{.2cm} &
    \hspace{.2cm}1\hspace{.2cm} &
    \hspace{.2cm}1\hspace{.2cm} &
    \hspace{.2cm}0\hspace{.2cm} &
    \hspace{.2cm}0\hspace{.2cm}\\

    \hspace{.2cm}0\hspace{.2cm} &
    \hspace{.2cm}0\hspace{.2cm} &
    \hspace{.2cm}0\hspace{.2cm} &
    \hspace{.2cm}0\hspace{.2cm} &
    \hspace{.2cm}0\hspace{.2cm} &
    \hspace{.2cm}0\hspace{.2cm} &
    \hspace{.2cm}1\hspace{.2cm} &
    \hspace{.2cm}1\hspace{.2cm}\\

    \hspace{.2cm}1\hspace{.2cm} &
    \hspace{.2cm}0\hspace{.2cm} &
    \hspace{.2cm}1\hspace{.2cm} &
    \hspace{.2cm}0\hspace{.2cm} &
    \hspace{.2cm}0\hspace{.2cm} &
    \hspace{.2cm}0\hspace{.2cm} &
    \hspace{.2cm}0\hspace{.2cm} &
    \hspace{.2cm}0\hspace{.2cm}\\

    \hspace{.2cm}0\hspace{.2cm} &
    \hspace{.2cm}1\hspace{.2cm} &
    \hspace{.2cm}0\hspace{.2cm} &
    \hspace{.2cm}1\hspace{.2cm} &
    \hspace{.2cm}0\hspace{.2cm} &
    \hspace{.2cm}0\hspace{.2cm} &
    \hspace{.2cm}0\hspace{.2cm} &
    \hspace{.2cm}0\hspace{.2cm}\\

    \hspace{.2cm}0\hspace{.2cm} &
    \hspace{.2cm}0\hspace{.2cm} &
    \hspace{.2cm}0\hspace{.2cm} &
    \hspace{.2cm}0\hspace{.2cm} &
    \hspace{.2cm}1\hspace{.2cm} &
    \hspace{.2cm}0\hspace{.2cm} &
    \hspace{.2cm}1\hspace{.2cm} &
    \hspace{.2cm}0\hspace{.2cm}\\

    \hspace{.2cm}0\hspace{.2cm} &
    \hspace{.2cm}0\hspace{.2cm} &
    \hspace{.2cm}0\hspace{.2cm} &
    \hspace{.2cm}0\hspace{.2cm} &
    \hspace{.2cm}0\hspace{.2cm} &
    \hspace{.2cm}1\hspace{.2cm} &
    \hspace{.2cm}0\hspace{.2cm} &
    \hspace{.2cm}1\hspace{.2cm}\\

    \hspace{.2cm}1\hspace{.2cm} &
    \hspace{.2cm}0\hspace{.2cm} &
    \hspace{.2cm}0\hspace{.2cm} &
    \hspace{.2cm}0\hspace{.2cm} &
    \hspace{.2cm}1\hspace{.2cm} &
    \hspace{.2cm}0\hspace{.2cm} &
    \hspace{.2cm}0\hspace{.2cm} &
    \hspace{.2cm}0\hspace{.2cm}\\

    \hspace{.2cm}0\hspace{.2cm} &
    \hspace{.2cm}1\hspace{.2cm} &
    \hspace{.2cm}0\hspace{.2cm} &
    \hspace{.2cm}0\hspace{.2cm} &
    \hspace{.2cm}0\hspace{.2cm} &
    \hspace{.2cm}1\hspace{.2cm} &
    \hspace{.2cm}0\hspace{.2cm} &
    \hspace{.2cm}0\hspace{.2cm}\\

    \hspace{.2cm}0\hspace{.2cm} &
    \hspace{.2cm}0\hspace{.2cm} &
    \hspace{.2cm}1\hspace{.2cm} &
    \hspace{.2cm}0\hspace{.2cm} &
    \hspace{.2cm}0\hspace{.2cm} &
    \hspace{.2cm}0\hspace{.2cm} &
    \hspace{.2cm}1\hspace{.2cm} &
    \hspace{.2cm}0\hspace{.2cm}\\

    \hspace{.2cm}0\hspace{.2cm} &
    \hspace{.2cm}0\hspace{.2cm} &
    \hspace{.2cm}0\hspace{.2cm} &
    \hspace{.2cm}1\hspace{.2cm} &
    \hspace{.2cm}0\hspace{.2cm} &
    \hspace{.2cm}0\hspace{.2cm} &
    \hspace{.2cm}0\hspace{.2cm} &
    \hspace{.2cm}1\hspace{.2cm}
\end{pmatrix}}$
        \end{tabular}
\vspace{.15cm}
\end{equation*}
and its kernel can be explicitly computed as
\begin{equation*}
    \ker(M^2) = \langle (1,-1,-1,1,-1,1,1,-1)\rangle.
\end{equation*}
Notice that the dimension of the kernel can also be obtained with the results provided in Section \ref{sec.random_coalescing_maps} and in particular through Equation \eqref{eq:R_all}, giving $m^n - R^{3,2}_2 = 8 - 7 =1$. Now, the first and last entries of any element in $p^*+ \ker M^2$ are given by 
\begin{equation*}
    (p^* + \ker M^2)_1 = -\frac{6}{60} +\alpha, \qquad  (p^* + \ker M^2)_8
    = - \alpha,
\end{equation*}
for some $\alpha \in \R$, showing it is impossible to have the first and last entry simultaneously non-negative. It follows that 
\begin{equation*}
(p^* +  \ker{M^2} )\cap \Delta = \emptyset,
\end{equation*}
hence the family $\{p(\ii_k, \cdot)\}_{\ii_k}$ is not observable.
\end{example}

\begin{remark}
    If a family of probabilities $\{p_{(\ii_k, \cdot)}\}_{\ii_k}$ is in the image via $M^k$ of a signed probability measure, then it is compatible (but not necessarily observable, as shown in Example \ref{ex:compatible-non-observable}).
\end{remark}

\textbf{Acknowledgements.} BBC is supported by the EPSRC Centre for Doctoral Training in Mathematics of Random Systems: Analysis, Modelling and Simulation (EP/S023925/1). FV is supported by the EPSRC (EP/S021590/1), the EPSRC Centre for Doctoral Training in Geometry and Number Theory (LSGNT), University College London, Imperial College London.

\bibliographystyle{plain}
\bibliography{main}

\begin{thebibliography}{10}

\bibitem{Babai1979}
L\'{a}szl\'{o} Babai.
\newblock Spectra of {C}ayley graphs.
\newblock {\em J. Combin. Theory Ser. B}, 27(2):180--189, 1979.

\bibitem{Baxendale1984}
Peter Baxendale.
\newblock Brownian motions in the diffeomorphism group. {I}.
\newblock {\em Compositio Math.}, 53(1):19--50, 1984.

\bibitem{daCosta2022}
Paulo~Henrique da~Costa, Michael~A. H\"{o}gele, and Paulo~R. Ruffino.
\newblock Stochastic n-point {D}-bifurcations of stochastic {L}\'{e}vy flows
  and their complexity on finite spaces.
\newblock {\em Stoch. Dyn.}, 22(7):Paper No. 2240021, 39, 2022.

\bibitem{Diaconis1981}
Persi Diaconis and Mehrdad Shahshahani.
\newblock Generating a random permutation with random transpositions.
\newblock {\em Z. Wahrsch. Verw. Gebiete}, 57(2):159--179, 1981.

\bibitem{Godsil2001}
Chris Godsil and Gordon Royle.
\newblock {\em Algebraic graph theory}, volume 207 of {\em Graduate Texts in
  Mathematics}.
\newblock Springer-Verlag, New York, 2001.

\bibitem{Knuth1970}
Donald~E. Knuth.
\newblock Permutations, matrices, and generalized {Y}oung tableaux.
\newblock {\em Pacific J. Math.}, 34:709--727, 1970.

\bibitem{Kunita1990}
Hiroshi Kunita.
\newblock {\em Stochastic flows and stochastic differential equations},
  volume~24 of {\em Cambridge Studies in Advanced Mathematics}.
\newblock Cambridge University Press, Cambridge, 1990.

\bibitem{Macdonald1995}
Ian~G. Macdonald.
\newblock {\em Symmetric functions and {H}all polynomials}.
\newblock Oxford Mathematical Monographs. The Clarendon Press, Oxford
  University Press, New York, second edition, 1995.

\bibitem{Qian2023}
Chengyang Qian, Yaokun Wu, and Yanzhen Xiong.
\newblock Inclusion matrices for rainbow subsets.
\newblock {\em Bull. Iran. Math. Soc.}, 50:1--65, 12 2023.

\bibitem{Robinson1938}
G.~de~B. Robinson.
\newblock On the {R}epresentations of the {S}ymmetric {G}roup.
\newblock {\em Amer. J. Math.}, 60(3):745--760, 1938.

\bibitem{Rockmore2002}
Dan Rockmore, Peter Kostelec, Wim Hordijk, and Peter~F. Stadler.
\newblock Fast {F}ourier transform for fitness landscapes.
\newblock {\em Appl. Comput. Harmon. Anal.}, 12(1):57--76, 2002.

\bibitem{Roichman1999}
Yuval Roichman.
\newblock Characters of the symmetric groups: formulas, estimates and
  applications.
\newblock In {\em Emerging applications of number theory ({M}inneapolis, {MN},
  1996)}, volume 109 of {\em IMA Vol. Math. Appl.}, pages 525--545. Springer,
  New York, 1999.

\bibitem{Sagan1990}
Bruce~E. Sagan and Richard~P. Stanley.
\newblock Robinson-{S}chensted algorithms for skew tableaux.
\newblock {\em J. Combin. Theory Ser. A}, 55(2):161--193, 1990.

\bibitem{Schensted1961}
C.~Schensted.
\newblock Longest increasing and decreasing subsequences.
\newblock {\em Canadian J. Math.}, 13:179--191, 1961.

\bibitem{Stanley1999}
Richard~P. Stanley.
\newblock {\em Enumerative combinatorics. {V}ol. 2}, volume~62 of {\em
  Cambridge Studies in Advanced Mathematics}.
\newblock Cambridge University Press, Cambridge, 1999.

\end{thebibliography}

\end{document}